\documentclass[10pt]{article}

\usepackage{bm}
\usepackage{geometry}                % See geometry.pdf to learn the layout options. There are lots.
\usepackage{graphicx}
\usepackage{amssymb}
\usepackage{epstopdf}
\usepackage{tikz}
\usepackage{color}
\usepackage{authblk}
\usepackage{hyperref}

\usepackage{amsfonts,latexsym,rawfonts,amsmath,amsthm, mathrsfs, lscape}
\usepackage[all]{xy}

\usepackage{array, tabularx}

\newtheorem{thm}{Theorem}[section]
\newtheorem{cor}[thm]{Corollary}
\newtheorem{lem}[thm]{Lemma}
\newtheorem{prop}[thm]{Proposition}

\theoremstyle{definition}
\newtheorem{defn}[thm]{Definition}
\theoremstyle{remark}
\newtheorem{rmk}[thm]{Remark}
\numberwithin{equation}{section}
\def \C {\mathbb C}

\def \D {\mathcal D}

\def \N {\mathbb N}

\def \Q {\mathbb Q}
\def \P {\mathbb P}

\def \O {\mathcal O}
\def \M {\mathcal M}

\def \X {\mathcal X}

\def \vol {\text{vol}}

\def \p {\partial}
\def \bp {\bar{\partial}}

\newcount\Comments  % 0 suppresses notes to selves in text
\usepackage{color}
\definecolor{darkgreen}{rgb}{0,0.5,0}

\newcommand{\kibitz}[2]{\ifnum\Comments=1\textcolor{#1}{#2}\fi}
\begin{document}
\title{Existence and deformations of K\"ahler-Einstein metrics on smoothable $\Q$-Fano varieties}
\author[1]{Cristiano Spotti \thanks{c.spotti@dpmms.cam.ac.uk}}
\author[2]{Song Sun\thanks{song.sun@stonybrook.edu. }}
\author[2]{Chengjian Yao \thanks{yao@math.sunysb.edu}}
\affil[1]{DPMMS University of Cambridge\\}
\affil[2]{Stony Brook University}

\renewcommand\Authands{ and }
\date{}

\maketitle

\begin{abstract} In this article we prove the existence of K\"ahler-Einstein metrics on $\Q$-Gorenstein smoothable, K-polystable  $\Q$-Fano varieties and we show how these metrics behave, in the Gromov-Hausdorff sense, under $\Q$-Gorenstein smoothings.  
\end{abstract}

%%%%%%%%%%%%%%
\section{Introduction}

It has been recently proved \cite{CDS0, CDS1, CDS2, CDS3} that the existence of a K\"ahler-Einstein (KE) metric on a smooth Fano manifold $X$ is equivalent to the algebro-geometric notion of ``K-polystability" (which is called \emph{K-stability} in \cite{CDS0}) .  Motivated by the study of the compactification of the moduli spaces of such smooth KE Fano manifolds, in this article we are going to investigate the existence of KE metrics on certain \emph{singular} Fano varieties and their local deformations. 

In order to justify our objects of interest, let us recall some facts on these KE moduli spaces.  Denote with $\M^n$ the space of all $n$-dimensional KE Fano manifolds, normalized so that the K\"ahler form $\omega$ is in the class $2\pi c_1(X)$, modulo biholomorphic isometries.  It is well-known that such set is pre-compact in the Gromov-Hausdorff topology. By \cite{DS} we can define the (refined) Gromov-Hausdorff compactification $\overline{\M}^n$ of $\M^n$: every point in the boundary $\overline{\M}^n\setminus \M^n$ is naturally a $\Q$-Gorenstein smoothable $\Q$-Fano variety which admits a K\"ahler-Einstein metric in a weak sense (see below for the precise definitions). It is  exactly this special category of singular Fano varieties, namely the ones which admit one-dimensional $\Q$-Gorenstein smoothings, on which we focus our attention, hoping that the understanding of the existence and deformations of these singular KE varieties will be useful in the study of structures  of the compactified moduli space $\overline{\M}^n$. Note it has been proved by Donaldson \cite{Do1} and Odaka \cite{Odaka} that being K\"ahler-Einstein (or being K-polystable) for smooth Fano manifolds with discrete automorphism group is a Zariski open condition. 
 
Before stating the main results of the paper, we recall some definitions. A {\em $\Q$-Fano variety} $X$ is a normal projective variety over $\C$ with at worst log-terminal singularities and with ample $\Q$-Cartier anti-canonical divisor $-K_{X}$. Thus, for some integer $\lambda$, $X$ is embedded into some projective space $\P^N$ by sections of $K_{X}^{-\lambda}$. A {\em weak K\"ahler-Einstein metric} on $X$ is a K\"ahler current in $2\pi c_1(X)$ with locally continuous potential, and that is a smooth K\"ahler-Einstein metric on  the smooth part $X^{reg}$ of $X$. Note there are different definitions in the literature, but they are all equivalent in this context \cite{EGZ}. 
A $\Q$-Fano variety $X$ is called {\em $\Q$-Gorenstein smoothable}\footnote{In general one can also define the notion of being \emph{smoothable}, without the condition on the relative anti-canonical divisor. This is a genuinely different notion, even for two dimensional quotient singularities, see \cite{KSB}} if there is a flat family  
$$\pi:\mathcal{X}\rightarrow \Delta,$$
over a disc $\Delta$ in $\C$ so that  $X\cong X_0$, $X_t$ is smooth for $t\neq 0$, and $\mathcal{X}$ admits a relatively $\Q$-Cartier anti-canonical divisor $-K_{\X/\Delta}$ (in this case $\pi:\mathcal{X}\rightarrow \Delta$ is called a \emph{$\mathbb{Q}$-Gorenstein smoothing} of $X_0$). 
It is well-known that, by possibly shrinking $\Delta$, we may assume that $X_t$ is a Fano manifold for $t\neq 0$, and that there exists an integer $\lambda>0$ such that $K_{X_t}^{-\lambda}$ are very ample line bundles with vanishing higher cohomology for all $t\in \Delta$. Moreover, the dimension, denoted by $N(\lambda)$, of the corresponding linear systems $|-\lambda K_{X_t}|$ is constant in $t$. 
Thus, when needed, we may assume  that the family $\mathcal{X}$ is relatively very ample, i.e. there is a smooth embedding $i:\mathcal{X}\rightarrow \mathbb{P}^{N(\lambda)}\times \mathbb{C}$  such that $i_t=i|_{X_t}:X_t\rightarrow \mathbb{P}^{N(\lambda)}\times\{t\}$ pulls the line bundle $\mathcal{O}(1)$ on $\mathbb{P}^{N(\lambda)}$ back to $K_{X_t}^{-\lambda}$.

The main theorem of this paper is the following result, which extends the results of \cite{CDS0, CDS1, CDS2, CDS3} to $\Q$-Gorenstein smoothable $\Q$-Fano varieties and simultaneusly gives some understanding on the way such singular metric spaces are approached by smooth KE metrics on the (analytically) nearby Fano manifolds.

\begin{thm}\label{thm1-1} Let $\pi:\mathcal{X}\rightarrow \Delta$ be a $\Q$-Gorenstein smoothing of a $\Q$-Fano variety $X_0$. If $X_0$ is K-polystable then $X_0$ admits a weak K\"ahler-Einstein metric $\omega_0$. 

Moreover, assuming that the automorphism group $Aut(X_0)$ is discrete, $X_t$ admit smooth K\"ahler-Einstein metrics $\omega_t$ for all $|t|$ sufficiently small and $(X_0,\omega_0)$ is the limit in the Gromov-Hausdorff topology of $(X_t, \omega_t)$, in the sense of \cite{DS}.
\end{thm}

Few remarks are in place. First, by the generalized Bando-Mabuchi uniqueness theorem \cite{BBEGZ} the above weak K\"ahler-Einstein metric $\omega_0$ is unique up to $Aut_0(X_0)$, the identity component of $Aut(X_0)$, thus can be viewed as a canonical metric on $X_0$. 
Second,  the above theorem does not just state that there is $\emph{a}$ sequence of nearby KE Fano manifolds which converges, in the Gromov-Hausdorff sense, to the weak KE metric, but that \emph{all} the nearby KE Fano manifolds actually converge to the unique singular limit $(X_0,\omega_0)$. Thus this property provides a good topological correspondence between complex analytic deformations (alias flat-families) and the notion of Gromov-Hausdorff convergence. For example, by the Bando-Mabuchi theorem we have the following immediate corollary:

\begin{cor}
Let $\pi: \mathcal X\rightarrow\Delta$ and $\pi': \mathcal X'\rightarrow\Delta$ be two $\Q$-Gorenstein smoothings of $\Q$-Fano varieties $X_0$ and $X_0'$. Suppose $X_t$ and $X_t'$ are bi-holomorphic for all $t\neq 0$, and $X_0$ and $X_0'$ are both K-polystable with discrete automorphism group, then $X_0$ and $X_0'$ are isomorphic varieties. 
\end{cor}

Notice that Corollary 1.2 would also follow from the separatedness of general polarized K-stable varieties, which is attempted purely algebraically by Odaka and Thomas \cite{OT}.

It is also worth noting that this kind of continuity at the boundary may be seen as an higher dimensional generalization in the KE case of the glueing results obtained by Spotti in \cite{Spotti} and Biquard-Rollin in \cite{BR}, but our proof here follows a very different approach. Another important point to remark, especially in view of applications to moduli spaces, is that the converse of our theorem, that is the fact that weak KE Fano varieties are indeed K-polystable was proved (without assuming the smoothability hypothesis) by Berman \cite{Ber2}.  Theorem \ref{thm1-1} should be also compared with the result on the existence of KE metrics and K-polystability on semi-log-canonical varieties with ample canonical $\mathbb{Q}$-bundle obtained respectively by Berman-Guenancia in \cite{BG} and by Odaka in \cite{Odaka1} and \cite{Odaka2}  (although the reason for the KE/K-polystability correspondence in this latter case is not fully clear).
Finally we remark that in a subsequent paper we will consider the very natural generalization  of Theorem \ref{thm1-1} to the case when $Aut(X_0)$ is not necessarily discrete with its connections to geometric invariant theory.

We now briefly describe the structure of this paper along with a sketch of the main arguments needed to prove our main Theorem \ref{thm1-1}. The strategy of the proof is based on the ``Donaldson's continuity method'' of deforming cone angle metrics along plurianticanonical divisors on the singular central fiber, a method which now makes the additional use of the better understood KE geometry in the nearby smooth fibers. In fact, one of the main issue in carring directly the deformation strategy on a singular variety is the lack of precise understanding of how a \emph{weak conical KE metric} (see definition below in the beginning of Section \ref{section2}) behaves near the singularities. For example, the elliptic theory required to prove ``openness'' is now missing. To overcome these difficulties our strategy consists, very roughly speaking, of studying the behavior on the singular fibers by approximating it with what happens in the nearby smooth varieties. The main argument uses several results which we now state and which, we think, may hold some interests in their own.

The first theorem, which is the content of Section $2$, shows that one can indeed construct a family of relative divisors for which the existence of weak conical KE metrics, with positive scalar curvature,  holds for angles close to the log-Calabi-Yau range. Moreover, an $L^{\infty}$-estimate (relative to the ambient Fubini-Study metric), which is important for later arguments, for the potentials of the \emph{conical K\"ahler-Einstein metrics} is obtained. 
The statement of the theorem reads as follows.

\begin{thm} \label{thm1-3}
Let $\pi:\mathcal{X}\rightarrow \Delta$ be a $\Q$-Gorenstein smoothing of a $\Q$-Fano variety, and suppose we have chosen an embedding $\X\subset \C\P^N\times\C$ (thus fixing a background Fubini-Study metric $\omega_{FS}$) as above.  Then, eventually shrinking $\Delta$, there exists a relative $\lambda$-plurianticanonical divisor $\mathcal{D}$,  smooth over $\Delta^*=\Delta\setminus\{0\}$, such that $(X_0, (1-\beta)D_0)$ is KLT for all $\beta\in (0, 1)$ (this is called a ``nice'' relative $\lambda$-plurianticanonical divisor in our context), and a  $\underline{\beta}\in (1-\lambda^{-1}, 1)$ such that for any $\beta \in (1-\lambda^{-1},\underline{\beta}]$, there is a weak conical K\"ahler-Einstein metric $\omega_{t, \beta}$ on $(X_t, (1-\beta)D_t)$ for all $t\in \Delta$, which is genuinely conical (in the sense of \cite{CDS1}) for $t\neq 0$ and such that $\omega_{t, \beta}=\omega_{t, FS}+\sqrt{-1}\p\bp\phi_{t,\beta}$ (here $\omega_{t, FS}=\frac{1}{\lambda}i_t^*\omega_{FS}$) with $|\phi_{t,\beta}|_{L^\infty}$ is uniformly bounded in $t$ for any fixed $\beta$.
\end{thm}

In Section \ref{section3} we show that the $L^\infty$-estimate obtained in the above theorem is indeed sufficient to prove that the family of conical KE metrics $\omega_{t,\beta}$ on $(X_t, (1-\beta) D_t)$  converges in the Gromov-Hausdorff sense as $t$ tends to zero (at the fixed angle $\beta$) to the weak conical KE metric $\omega_{t,\beta}$ on $(X_0,(1-\beta)D_0)$. 

This last Gromov-Hausdorff continuity result is the starting point of the actual continuity type argument which we run in Section \ref{section4}. The first preliminary theorem is the following slightly technical, but useful, result.
 Given a $\Q$-Gorenstein smoothing with a ``nice'' $\lambda$-anticanonical divisor $(\mathcal{X},\mathcal{D}) \rightarrow \Delta$ we define the function
$$\beta_t:= \sup\{\beta\in (1-\lambda^{-1},1] \,|\, \mbox{there exists a weak conical KE with cone angle } 2\pi\beta \mbox{ on } (X_t,D_t)\}.$$
Note that by Theorem \ref{thm1-3} proved in Section \ref{section2} we know that, eventually shrinking the disc, $\beta_t$ is uniformly bounded from below by $\underline{\beta} > 1-\lambda^{-1}$. Besides, it satisfies \emph{lower-semi-continuity}: 
\begin{thm}\label{thm1-4}
$\beta_t$ is a lower-semi-continuous function of $t \in \Delta$.
\end{thm}

Actually, since a-priori we do not know that the set of angle values of weak conical KE metrics on a singular variety is connected, we will first prove the above theorem for a slight variation of the previously defined function. However, in the end the statement turns out to be true, so we avoid further going into such technicality in this introduction (but see Section \ref{section4} for more precise definitions and statements). 

With these results in hands, we are now able to perform  the desired ``Donaldson's continuity method'', that is to prove \emph{openness} and \emph{closedness} of the set of angles for \emph{weak conical KE metrics} on the central fiber under the hypothesis of K-polystability. The openness part makes use of the idea in \cite{Yao} and some of the topological arguments required are somehow reminiscent of the ones in Odaka-Spotti-Sun \cite{OSS}. A discussion on the needed definition of refined GH topology is given along the proofs. 
We remark that the \emph{weak KE metrics} on the central fiber is constructed as, roughly speaking, a diagonal limit of conical KE metrics on smooth pairs $(X_t, (1-\beta(t))D_t)$, where, as the complex parameter $t$ tends to zero, the cone angles $2 \pi \beta(t)$ open up to $2\pi$. Under the additional  hypothesis of discreteness of $Aut(X_0)$ one can easily observe that the same reasonings actually imply that $\beta_t$ is identically equal to $1$ and it is actually achieved as a maximum. This will conclude the proof of our main Theorem \ref{thm1-1}.

In the appendix we include a proof of the existence of ``nice'' relative anticanonical divisors required in our arguments.

 While we are finishing this paper, we learned that Chi Li, Xiaowei Wang, and Chenyang Xu \cite{LWX} have independently proved a result that is closely related to the results obtained here. Their method is more algebraic and the results are not exactly the same. We feel the combination of the two results would yield a better understanding of the structure of the compactified moduli space of K\"ahler-Einstein Fano manifolds.

\subsection*{Acknowledgements}
C. S.  carried part of this research under a ``Carmin fellowship'' at IHP $\&$ IHES and a grant from the project ANR-10-BLAN 0105 at the ENS Paris. He would in particular thank Professor Olivier Biquard for several stimulating discussions. S. S.  is partly supported by NSF grant 1405832 and Alfred P. Sloan research fellowship, and by the Simons Center for Geometry and Physics during the year 2013-2014. C-J. Y. is very grateful to his advisor Xiuxiong Chen for constant support and encouragement. We all would like to thank Professor Eric Bedford and Dr. Long Li for valuable discussions on some points, and Yuji Odaka for helpful comments on the draft of this paper.

\section{Small angle existence and  $L^\infty$-bound on the potentials}\label{section2}

We first recall the definitions of the various relevant functionals in K\"ahler geometry which we use in this section. Let $X$ be a  smooth Fano manifold, and $D$ be a smooth divisor in $|-\lambda K_X|$ for some $\lambda>1$ with defining section $S$. Given a smooth K\"ahler metric $\omega\in 2\pi c_1(X)$, then it is the curvature form of  a Hermitian metric $h$ on $K_X^{-1}$ (defined up to a constant multiple). A choice of $h$ can also be viewed as a choice of a smooth volume form $\text{vol}_h$ on $X$, or equivalently, a choice of the \emph{Ricci potential} $h_\omega$ of $\omega$ (that is to say $\text{Ric }\omega=\omega+\sqrt{-1}\p\bp h_\omega$ by the relation $\vol_h =e^{h_\omega}\frac{\omega^n}{n!}$). These can be extended to the more general case when $X$ is a $\mathbb Q$-Fano variety, see \cite{EGZ}. For example, the notion of \emph{smooth K\"ahler metrics} on $X$ is still well-defined, under which the Fubini-Study metric restricts to is a smooth K\"ahler metric if $X$ is embedded in projective space.

Denote by $PSH(X, \omega)$ the space of $\omega$-plurisubharmonic functions on $X$. 
For $\beta\in (0, 1]$ we define $r(\beta)=1-(1-\beta)\lambda$ and denote $V=(2\pi)^n \frac{<c_1(X)^n,[X]>}{n!}$ (this is a fixed topological quantity all through this paper). 

We first remark that we have  different notions of being K\"ahler-Einstein depending on the situation we are considering (we refer to \cite{CDS3} Section 4, for the precise definitions): a \emph{smooth K\"ahler-Einstein metric} is the standard one on a smooth manifold; a (genuinely) \emph{conical K\"ahler-Einstein metric} is defined for a pair $(X, (1-\beta)D)$ of a smooth K\"ahler manifold and a smooth divisor (called K\"ahler-Einstein metrics with cone singularities in \cite{Do}); a \emph{weak K\"ahler-Einstein metric} is defined on a $\Q$-Fano variety $X$ \cite{EGZ}; a \emph{weak conical K\"ahler-Einstein metric} is defined on a KLT pair $(X, (1-\beta)D)$ (see \cite{BBEGZ}). Accordingly, we have the notion of \emph{K-polystability} for such a KLT pair, see \cite{CDS3}.

In particular, the \emph{weak conical K\"ahler-Einstein metric} used in this paper means a K\"ahler current $\omega_\phi=\omega+\sqrt{-1}\p\bp \phi$ (where $\omega$ is a smooth K\"ahler metric on $X$, and $\phi$ is smooth on $X^{reg}\backslash D$ and locally continuous near $X^{sing}\cup D$), which satisfies the current equation:
\begin{equation}\label{eqn2-1}
\text{Ric } \omega_\phi=r(\beta)\omega_\phi+2\pi(1-\beta) [D]
\end{equation}
 in a suitable sense or, equivalently, the complex Monge-Amp$\grave{\text{e}}$re equation
\begin{equation}\label{eqn2-2}
\omega_\phi^n=n!V\frac{e^{-r(\beta)\phi} |S|_h^{2\beta-2}\text{vol}_h}{\int_X e^{-r(\beta)\phi} |S|_h^{2\beta-2}\text{vol}_h}.
\end{equation}
where $S$ is the defining section of $D$, and $h$ on $L_D$ is the natural Hermitian metric induced from the one on $K_{X}^{-1}$.

We should remark that in the case when $(X, D)$ is a smooth pair, the \emph{weak conical K\"ahler-Einstein metric} is actually genuinely conical by the work of \cite{CGP, GP, JMR, LS, Y1, CW}. 

Recall the various functionals that are needed, see \cite{LS} for a collection of them. From now on, fix a smooth K\"ahler metric on the $\mathbb{Q}$-Fano variety $X$ and fix a Hermitian metric $h$ on the $\mathbb{Q}$-line bundle $K_X^{-1}$.

%%%%%%%%---------------------------
% Definition of log-Ding-functional
%%%%%%%%---------------------------
\begin{defn}[log-Ding-functional]\label{log-Ding-functional}  For $\phi\in L^\infty(X)\cap PSH(X,\omega)$, we define
$$F_{\omega,(1-\beta)D}(\phi):=F_{\omega}^0(\phi)+F_{\omega}^1(\phi), $$ 
where 
$$F_{\omega}^0(\phi):=-\frac{1}{(n+1)!}\sum_{i=0}^n \int_{X} \phi(\omega+\sqrt{-1}\partial\bar\partial \phi)^i\wedge\omega^{n-i},$$
$$F_\omega^1(\phi):=-\frac{V}{r(\beta)}\log \frac{1}{V}\int_{X} e^{-r(\beta)\phi}|S|_h^{2\beta-2}\text{vol}_h.$$
The first term is well-defined by the pluri-potential theory \cite{BBEGZ} and the second one makes sense when $(X, (1-\beta)D)$ is KLT. The critical point equation for this \emph{log-Ding-functional} is precisely the above Equation \ref{eqn2-2}.
\end{defn}

Under the assumption that $(X,D)$ are \emph{smooth}, we could further define the \emph{log-Mabuchi-functional}. Firstly, define $ C^{1,1}(X):=\bigcup_{\beta\in(0,1)}C^{1, 1}_\beta(X)$, where $C^{1, 1}_\beta(X)$ denotes the space of all functions $\phi$ which is $C^2$ on $X^{reg}\backslash D$ with $\omega+i\p\bp\phi\geq 0$, and locally around each point $p$ in $D$, we have $-C\omega_{(\beta)}\leq \omega+i\p\bp \phi\leq C\omega_{(\beta)}$ for some $C>0$, where $\omega_{(\beta)}$ is a standard model conical K\"ahler metric in a neighborhood of $p$. 

%%%%%-------------------------------------
% Definition of log-Mabuchi-functional
%%%%%-------------------------------------

\begin{defn}[log-Mabuchi-functional]\label{def2-2}
Suppose that $(X, D)$ is smooth and $\omega$ is smooth, let 
$$\text{Ric }\omega=r(\beta)\omega+2\pi (1-\beta)[D]+\sqrt{-1}\p\bp H_{\omega, (1-\beta)D}$$
where $H_{\omega, (1-\beta)D}=h_\omega-\log |S|_h^{2-2\beta}$. For $\phi\in PSH(X, \omega)$ which is in $ C^{1, 1}(X)$,
 we define
 $$M_{\omega,(1-\beta)D}(\phi):=\int_{X} \log \frac{\omega_\phi^n}{e^{H_{\omega, (1-\beta)D}}\omega^n}\frac{\omega_\phi^n}{n!}
-r(\beta)(I-J)_\omega(\phi)+\int_{X} H_{\omega, (1-\beta)D} \frac{\omega^n}{n!}$$
where
$$I_\omega(\phi):=\int_X \phi\frac{(\omega^n-\omega_\phi^n)}{n!},$$
$$J_\omega(\phi):=\int_X \phi\frac{\omega^n}{n!}-\frac{1}{(n+1)!}\sum_{i=0}^n\int_X \phi(\omega+\sqrt{-1}\partial\bar\partial\phi)^i\wedge\omega^{n-i}.$$
\end{defn}

Both functionals can be defined for more general class of functions, see \cite{Ber1}, \cite{BBEGZ}. Here we only deal with the spaces that suit our later purpose. 
There is a simple relation between these two functionals, originally due to Tian, which is needed later. 

\begin{prop}[\cite{Ber1},\cite{Tian}] \label{logDM}
$$M_{\omega,(1-\beta)D}(\phi)\geq r(\beta)F_{\omega, (1-\beta)D}(\phi)+\int_X H_{\omega, (1-\beta)D}\frac{\omega^n}{n!}$$
\end{prop}

From now on we return to the setting of Theorem \ref{thm1-1}. Namely we consider a $\mathbb{Q}$-Gorenstein smoothing $\mathcal{X}$ of a $\mathbb{Q}$-Fano variety $X_0$. We further assume the existence of a  ``nice'' divisor $\mathcal{D}$ possessing the property claimed in the first half of Theorem \ref{thm1-3} (see the Appendix for a proof). 

In the next paragraphs we are going to show the existence in family of (weak) conical KE metrics for sufficiently small (but uniform) values of the cone angles.

\subsection{An alpha-invariant estimate and uniform small angle of existence}

Let $X$ be a $\mathbb{Q}$-Fano variety, and $D\in|-\lambda K_X|$ for $\lambda>1$ such that $(X, (1-\beta)D)$ is KLT for all $\beta\in (0,1]$, let $\omega\in 2\pi c_1(X)$ be a smooth K\"ahler metric on $X$. Recall the definition of the generalized alpha-invariant \cite{BBEGZ}: 
$$\alpha(K_X^{-1}, (1-\beta)D):=\sup\{\alpha>0| \sup_{\sup \phi=0; \omega+\sqrt{-1}\partial\bar\partial \phi\geq 0} 
\int_{X} e^{-\alpha \phi}|S|_h^{2\beta-2} \text{vol}_h <+\infty\},$$
 
When $\beta=1$, we denote this invariant simply by $\alpha(K_X^{-1})$. An easy H\"older's inequality argument shows the following, see also Remark 2.25 in \cite{LS}:

\begin{prop} \label{prop2-4}
 $\alpha(K_{X}^{-1}, (1-\beta)D)\geq \beta\alpha(K_X^{-1})>0$.
\end{prop}

\begin{proof}
The last inequality is proved in \cite{BBEGZ}, Proposition 1.4. By the H\"older's inequality, for $\phi$ with $\sup \phi=0$, we have 
$$\int_{X} e^{-\alpha \phi}|S|_h^{2\beta-2}\text{vol}_h
\leq \left(\int_X e^{-\alpha p\phi} \text{vol}_h\right)^\frac{1}{p}\left(\int_X |S|_h^{(2\beta-2)q}\text{vol}_h\right)^\frac{1}{q},$$
where we take  $q=\frac{1}{1-\beta+\epsilon \beta}$, and $p=\frac{1}{\beta-\epsilon\beta}$ for $\epsilon>0$ small.  The quantity 
$$\int_X |S|_h^{(2\beta-2)q}\text{vol}_h$$
is bounded by the KLT condition and, if we take $\alpha=\tilde\alpha (1-\epsilon)\beta$ for any $\tilde\alpha<\alpha(K_X^{-1})$, also the term 
$$ \int_X e^{-\alpha p\phi}\text{vol}_h=\int_X e^{-\tilde\alpha\phi}\vol_h$$
is uniformly bounded above by $C=C({\tilde\alpha})$, by the definition of the ordinary alpha-invariant. 

Therefore the estimate in the proposition follows, since $\alpha(K_X^{-1}, (1-\beta)D)\geq \tilde\alpha (1-\epsilon)\beta$ for any $\tilde\alpha<\alpha(K_{X}^{-1})$ and any small $\epsilon>0$.
\end{proof}

By extending Tian's result relating alpha-invariant and the existence of K\"ahler-Einstein metrics, it has been proved in \cite{BBEGZ} that
\begin{prop}[\cite{BBEGZ}] \label{central fiber solvable}
Suppose that $$\alpha(K_X^{-1}, (1-\beta^\ast) D)> r(\beta^\ast) \frac{n}{n+1},$$ for some $\beta^\ast> 1-\lambda^{-1}$, where $r(x)=1-(1-x)\lambda$. Then, for any $\beta\in (0, \beta^\ast]$, there is a weak conical K\"ahler-Einstein metric on $(X, (1-\beta)D)$. 
\end{prop}

Note that in \cite{BBEGZ}, the alpha-invariant of the pair has a different normalization from the above one. The condition in the previous Proposition \ref{central fiber solvable} is just the condition in Proposition 4.13 in \cite{BBEGZ} in our conventions. 

Consider a $\Q$-Gorenstein smoothing $\mathcal{X}\rightarrow \Delta$ of a $\Q$-Fano variety $X_0$, equipped with a family of nice relative $\lambda$-plurianticanonical divisors $D_t$  as before. Since $r(\beta)=1-(1-\beta)\lambda$ tends to zero as $\beta$ tends to $1-\lambda^{-1}$ and $\lambda>1$,  we obtain as an immediate corollary of the above two propositions the following statement on the central fiber $X_0$:

\begin{cor}\label{cor2-7}
There exists some $\beta'>1-\lambda^{-1}$ so that for all $\beta\in (0, \beta']$, there exists a weak conical K\"ahler-Einstein metric on $(X_0, (1-\beta)D_0)$. 
\end{cor}

\begin{rmk} \label{rmk2-7}
As is shown in \cite{BBEGZ}, Theorem 5.4, the above corollary implies that the identity component $Aut_0(X_0, D_0)$ of the automorphism group consists of just the identity. This fact will be used later. 
\end{rmk}

On the other hand, on the smooth fibers $X_t$, $t\neq 0$,  we have the following uniform estimates by an observation of Odaka: 

\begin{prop}[\cite{Odaka}]
There exits $\beta''>1-\lambda^{-1}$, so that $(X_t, (1-\beta)D_t)$ is K-polystable for all $t\in\Delta^*$ and $\beta\in (0, \beta'']$.
\end{prop}

We briefly describe a proof here using the relation between the analytically defined alpha-invariant with the algebraically defined \emph{global log canonical threshold} \cite{CSD},

$$\alpha(X_t)=\text{glct}(X_t, K_{X_t}^{-1})$$
where for a polarized manifold $(X, L)$, 
\begin{eqnarray*}
\text{glct}(X, L)&=&\inf_{k>0}\inf_{E\in |kL|} \text{lct}(X, k^{-1}E)\\
&:=&\inf_{k>0}\inf_{E\in |kL|)}\{\sup_{c>0}|(X, ck^{-1}E)\ \text{is log canonical}\}.
\end{eqnarray*}

By \cite{Mus},  $\text{lct}(X, k^{-1}E)\geq k(mult(E))^{-1}$. We can apply this to $X_t$: since $X_t$ is embedded into $\P^N$ by $|-\lambda K_{X_t}|$, it follows that for any $E\in |k(-\lambda K_{X_t})|$, $mult(E)\leq Ck$, for constant $C>0$ independent of $t$ and $k$. Therefore $\text{glct}(X_t, K_{X_t}^{-\lambda})$ has a uniform positive lower bound. On the other hand, by the scaling relation,
 $\text{glct}(X_t, K_{X_t}^{-1})=\lambda  \text{ glct}(X_t, K_{X_t}^{-\lambda})$, 
is also bounded from below since the power $\lambda$ is uniform for all $X_t$ in $\mathcal{X}$. 

We then apply Proposition \ref{prop2-4} and \ref{central fiber solvable}, there exists $\beta''>1-\lambda^{-1}$ such that there exists \emph{weak conical K\"ahler-Einstein metrics} on $(X_t, (1-\beta)D_t)$ for $t\in \Delta^*, \beta\in (0,\beta'']$. According to \cite{Ber2}, $(X_t,(1-\beta)D_t)$ is K-polystable for $t\in \Delta^*, \beta\in (0,\beta'']$.

By \cite{CDS1, CDS2, CDS3}, there exists a unique \emph{genuine conical K\"ahler-Einstein metric} $\omega_{t,\beta}$ on $(X_t, (1-\beta)D_t)$ for $\beta\in (0, \beta'']$ and $t\in \Delta^*$. Actually these \emph{conical K\"ahler-Einstein metrics} concide with the above \emph{weak conical K\"ahler-Einstein metrics} obtained by \cite{BBEGZ} by the regularity result \cite{GP, CW} and the ``uniqueness'' from \cite{SW} and \cite{BBEGZ}. In summary, take $\underline{\beta}=\min(\beta', \beta'')>1-\lambda^{-1}$, we have 

\begin{cor} \label{cor3-9}
For any $\beta\in (0, \underline{\beta}]$ and $t\in \Delta$, there is a unique weak conical K\"ahler-Einstein metric on $(X_t, (1-\beta)D_t)$, which is  of positive scalar curvature for $\beta\in (1-\lambda^{-1},\underline{\beta}]$ and genuinely conical when $t\neq 0$.  
\end{cor}

From this, to prove theorem \ref{thm1-3}, it suffices to establish the uniform $L^\infty$ estimate of the conical K\"ahler-Einstein potentials.  We remark that by standard theory this follows from a ``uniform properness'' of the log-Mabuchi-functional as $t$ varies, and the latter would be straightforward if we can derive a uniform upper bound (independent of $t$) of the integral in the definition of alpha-invariant.  This is certain type of semi-continuity property which seems to be correct in the end, but we do not know a proof of that, so we need to get around with this by an \emph{ad hoc} method. 

\subsection{Uniform bounds on energy functionals}
Suppose  we have a family $(\X, \D)$ over $\Delta$, which is embedded into $\P^N\times \C$ by $K_{\X/\Delta}^{-\lambda}$. We denote by $\omega_{t, FS}$ the restriction of $\lambda^{-1}\omega_{FS}$ on $X_t$, and by $h_t$ the restriction of $h_{FS}^{1/\lambda}$ on $K_{\X/\Delta}|_{X_t}$,  $\beta\in (0, 1)$, we define $F_{t,\beta}$ (respectively $M_{t,\beta}$) to be the infimum of the log-Ding-functional  (respectively log-Mabuchi-functional) on $X_t$ with base point $\omega_t=\omega_{t, FS}$. These are finite when $X_t$ admits a weak conical K\"ahler-Einstein metric according to Theorem 4.8 of \cite{BBEGZ}, and $F_{t,\beta}=F_{\omega_t, (1-\beta)D_t}(\phi_{t,\beta})$ is achieved at the unique conical KE metric $\omega_{t,\beta}=\omega_{t}+\sqrt{-1}\partial\bar\partial \phi_{t,\beta}$.

\begin{prop} \label {prop2-10}
Suppose for a fixed $\beta\in (1-\lambda^{-1}, 1)$, or $\beta=1$ if $Aut(X_0)$ is discrete,  there are weak conical K\"ahler-Einstein metrics $\omega_{t,\beta}$ on $(X_t, (1-\beta)D_t)$ for all $t\in \Delta$, which are genuinely conical for $t\neq 0$. Then we have that $\limsup_{t\rightarrow0} F_{t,\beta}>-\infty$ and $\limsup_{t\rightarrow 0}M_{t,\beta}>-\infty$. 
\end{prop}

\subsubsection{Lower bound of log-Ding-functional}
We first prove the statement about the log-Ding-functional. The idea  comes from \cite{LS}. For $r\in (0, 1)$, we denote by $\mathcal{X}_r=\mathcal{X}|_{\Delta_r}\subset \mathbb{P}^N\times \Delta_r$. $\mathcal{X}_r$ is viewed as a complex analytic variety with smooth boundary, endowed with a natural K\"ahler metric $\Omega=\lambda^{-1}\omega_{FS}+\sqrt{-1}\mathrm{d}t\wedge\mathrm{d}\bar t$.  Let $\omega_{t,\beta}=\lambda^{-1}\omega_{FS}+\sqrt{-1}\partial\bar\partial \phi_{t, \beta}$ be the (weak) conical K\"ahler-Einstein metric on $X_t$ for $\beta\in (1-\lambda^{-1}, 1)$ as in the Proposition (the case for $\beta=1$ and discrete $Aut(X_0)$ is simpler). By \cite{BBEGZ} and \cite{SW} such metric is unique, and we use the natural normalization of $\phi_{t,\beta}$ determined by the Monge-Amp$\grave{\text{e}}$re equation 
$$(\omega_t+\sqrt{-1}\partial\bar\partial\phi_{t,\beta})^n=e^{-r(\beta)\phi_{t,\beta}}|S_t|_{h_t}^{2\beta-2}{\text{vol}_{h_t}}$$ 
and then define a function $\Psi(t,\cdot)=\phi_{t,\beta}(\cdot)$ on $\mathcal{X}_r$. Now fix $r\in(0, 1)$, we want to solve the Dirichlet problem for the following homogeneous complex Monge-Amp\`ere equation

 \begin{equation} \label{eqn2-3}  \left\{ \begin{array}{ll}
               (\Omega+\sqrt{-1}\partial\bar\partial \Phi)^{n+1}=0;\\
               \Omega+\sqrt{-1}\partial\bar\partial \Phi\geq 0\\
               \Phi|_{\partial \mathcal{X}_r}=\Psi.
               \end{array}\right.
               \end{equation}

Since $\Psi$ is not smooth, due to the fact that $\phi_{t,\beta}(\cdot)$ is the potential for conical KE metrics, we will instead solve the equation for some smooth approximations to the boundary conditions.  A geometrically natural approximation $\omega_{t,\beta}^\epsilon=\omega_t+\sqrt{-1}\partial\bar\partial \phi_{t,\beta}^\epsilon$ is the solution of the following Monge-Amp$\grave{\text{e}}$re equation
\begin{equation} \label{eqn2-4}
(\omega_t+\sqrt{-1}\partial\bar\partial \phi_{t,\beta}^\epsilon)^n=e^{-r(\beta)\phi_{t,\beta}^\epsilon}\frac{\vol_{h_t}}{(|S_t|_{h_t}^2+\epsilon)^{1-\beta}}.
\end{equation}
\noindent For $\epsilon\in (0,1]$ and $t\neq 0$, this equation can be solved by following a continuity path, see \cite{CDS1}. Let $\Psi^\epsilon(t, z)=\phi_{t,\beta}^\epsilon(z)$. As is shown in \cite{CDS1}, for any fixed $t\in \Delta^*$, as $\epsilon\to 0$ the potential $\Psi^\epsilon(t,\cdot)$ converges to $\phi_{t,\beta}$ globally in $C^\alpha$ sense on $X_t$ and  in $C^\infty$ sense on any compact set away from $D_t$. For our purpose we need a uniform convergence (as $\epsilon\to 0$) for $t$ with $|t|=r>0$ fixed. 

First recall that the oscillation of a K\"ahler potential is controlled by the $I$ functional together with the Poincar\'e and Sobolev constants bound, see \cite{Tian}. %%%

\begin{lem} \label{lem2-11}
Let $\omega$ and $\omega_\phi=\omega+\sqrt{-1}\partial\bar\partial\phi$ be two smooth K\"ahler metrics on a compact K\"ahler manifolds $X$. Then 
$$\text{osc } \phi\leq \{C_S(\omega)^{\delta_n}C_P(\omega)+C_S(\omega_\phi)^{\delta_n}C_P(\omega_\phi)\} I_\omega(\phi)+C_S(\omega)^{\delta_n}+C_S(\omega_\phi)^{\delta_n},$$
\noindent where $C_P, C_S$ are the Poincar$\acute{\text{e}}$ and Sobolev constants of the corresponding metrics on $X$ and $\delta_n$ is a dimensional constant. 
\end{lem}

\noindent \emph{Outline of Proof}: The application of Moser's iteration to the inequality $\Delta_\omega(\phi-\frac{1}{V}\int_X \phi\omega^n)>-n$ and $\Delta_{\omega_\phi} (-\phi+\frac{1}{V}\int_X \phi\omega_\phi^n)>-n$   gives the control of $\sup\phi-\frac{1}{V}\int_X \phi\omega^n$ and $\inf \phi-\frac{1}{V}\int_X \phi\omega_\phi^n$, by the $L^2$ norm of $\phi$ under the two metrics $\omega, \omega_\phi$, quantities which are then bounded by the $I$ functional. The Poincar\'e and Sobolev bounds appear in the inequalities used above. \qed

We need the following comparison between the conical KE metric and the Fubini-Study metric:

\begin{lem} \label{lem1-6}
There exists a constant $C$ depending only on $\text{Osc }\phi_{t,\beta}$ so that 
$$\omega_{t, \beta}\geq C^{-1} \omega_{t, FS}$$ 
for $t\in \Delta^*$. 
\end{lem}

\begin{proof}
This follows from the Chern-Lu inequality, as used in \cite{JMR}, \cite{CDS2}. For convenience we provide an argument here. We assume $\beta<1$, and the argument for $\beta=1$ is simpler. As already explained above,  by \cite{CDS1}, $\omega_{t, \beta}$ can be approximated by a family of smooth K\"ahler metrics $\omega_{t, \beta}^\epsilon=\omega_{t, FS}+i\p\bp \phi_{t, \beta}^\epsilon$ on $X_t$ with positive Ricci curvature. Thus it suffices to prove the conclusion for $\omega_{t, \beta}^\epsilon$ uniformly. Let $e=\text{tr}_{\omega_{t, \beta}^\epsilon} \omega_{t, FS}$.  By a direct calculation we have for $A>0$ that  
$$\Delta_{\omega_{t,\beta}^\epsilon}(\log e-A\phi_{t, \beta}^\epsilon)\geq -nA+(A-C_2)e, $$
where $C_2$ is  the upper bound of bisectional curvature of $\omega_{t, FS}$. By standard Hermitian differential geometry, it follows that $C_2$ is bounded by the bisectional curvature of $\P^N$. Choose $A=C_2+1$, by maximum principle we obtain $e\leq C$ for $C$ depending only on $\text{Osc }\phi_{t, \beta}^\epsilon$. Let $\epsilon\rightarrow0$, we have $\phi_{t, \beta}^\epsilon$ converges in $L^\infty$ to $\phi_{t, \beta}$ on $X_t$ and the convergence is smooth away from the divisor $D_t$. The conclusion then follows. 
\end{proof}

 \begin{lem} \label{lem2-12}
 For any $r,A>0$, $\delta>0$ sufficiently small and $K\subset\subset (\mathcal X\setminus \mathcal D)|_{\p \Delta_r}$ and any $k>0$, there is a constant $C>0$ depending only on $r$, $A$, $\delta$, $K$ and $k$ so that for all the solutions $\phi_{t, \beta}^\epsilon$ of Equation (\ref{eqn2-4}) with $\beta>1-\lambda^{-1}+\delta$ and with $I_{\omega_t}(\phi_{t, \beta}^\epsilon)\leq A$ we have 
 $||\phi_{t,\beta}^\epsilon||_{C^k(K)}\leq C$.
\end{lem}
 
\begin{proof}
The background Fubini-Study metric has uniform Poincar$\acute{\text{e}}$ and Sobolev constant by \cite{DT}, while the K\"ahler metric $\omega_{t,\beta}^\epsilon$ also has uniform Poincar$\acute{\text{e}}$ and Sobolev constant since it has a uniform positive lower bound $r(\beta)$ on the Ricci curvature. By Lemma \ref{lem2-11} the oscillation $\text{Osc }\phi_{t,\beta}^\epsilon$ is bounded by the $I$ functional. Then, by the above Lemma \ref{lem1-6} we have a uniform lower bound of the form $\omega_{t, \beta}^\epsilon>C^{-1}\omega_{t}$.  Away from the singularities $X_0^{sing}$ and the divisor $\D$, since $\frac{\vol_{h_t}}{\omega_t^n}$ is uniformly smooth in any uniform local holomorphic coordinate ball, we get the uniform upper bound $\omega_{t,\beta}^\epsilon<C\omega_t$. Together with Equation (\ref{eqn2-4}) with get that $|\phi_{t,\beta}^\epsilon|_{L^\infty}$ is uniformly bounded and then the standard elliptic $W^{2,p}$ estimate could yield us uniform $C^{1,\alpha}$ bound on $\phi_{t,\beta}^\epsilon$.  Then the standard Evans-Krylov theory \cite{Evans, Kry} for the complex Monge-Amp$\grave{\text{e}}$re equation could bootstrap to higher order bound on the KE potentials $\phi^\epsilon_{t,\beta}$.  \end{proof}
 
The next lemma shows that the $I$ functional is continuous under the above continuity of K\"ahler potentials. 
 
 \begin{lem} \label{lem2-13}
 Suppose we have $t_j\rightarrow t_0\in \Delta$, and suppose we have a sequence of potentials $\phi_j$ on $X_{t_j}$ with $\omega_{t_j}+\sqrt{-1}\p\bp \phi_j\geq 0$ and $|\phi_{t_j}|_{L^\infty}$ is uniformly bounded. Furthermore assume $\phi_j$ is $C^2$ on $X_{t_j}\setminus D_{t_j}$, and $\phi_j$ converges smoothly away from $\mathcal{X}^{sing}\cup \mathcal{D}$.  Then we have 
$$\lim_{j\rightarrow\infty} I_{\omega_{t_j}}(\phi_j)=I_{\omega_{t_0}}(\phi_0). $$
 \end{lem}
 \begin{proof}
 We assume $t_0=0$. The other case is simpler. We write 
 $$I_{\omega_{t_j}}(\phi_j)=\int_{U_{t_j}} \phi_j(\omega_{t_j}^n-(\omega_{t_j}+i\p\bp\phi_j)^n)+\int_{X_{t_j}\setminus U_{t_j}} \phi_j(\omega_{t_j}^n-(\omega_{t_j}+i\p\bp\phi_j)^n)$$
 It is clear that the first term converges to zero. For the second term, we notice that we can choose $U_0$ in $X_0$ so that $\int_{X_0\setminus U_0} \omega_{0}^n$ and $\int_{X_0\setminus U_0} (\omega_0+i\p\bp\phi_0)^n$ arbitrarily small since the complement is arbitrarily close to the volume of $X_0$. 
 \end{proof}
 
\begin{lem}\label{lem2-14}
For any $t_0\in \Delta^*$, there exists $\delta_{t_0}>0$, such that  for any $t\in B_{\delta_{t_0}}(t_0)$ and $\epsilon \in (0,\delta_{t_0}]$
$$|I_{\omega_t}(\phi_{t,\beta}^\epsilon)-I_{\omega_{t_0}}(\phi_{t_0,\beta}^\epsilon)|\leq 1.$$ 
 \noindent Consequently by Lemma \ref{lem2-12} for any $K\subset\subset \Delta^*$ there exists $C_K>0$ and $\delta_K>0$ such that 
 $$\sup_{\epsilon\in (0, \delta_K], t\in K}||\phi_{t,\beta}^\epsilon||_{L^\infty}\leq C_K.$$ 
 \end{lem}

\begin{proof}
For $t_0\in \Delta^*$, suppose the bound does not hold. Then we could pick a sequence $t_j\to t_0$ and $\epsilon_j\to 0$, such that $|I_{\omega_{t_j}}(\phi_{t_j,\beta}^{\epsilon_j})-I_{\omega_{t_0}}(\phi_{t_0,\beta}^{\epsilon_j})|>1$. By the continuous dependence on $t$ of these quantities for any fixed $\epsilon>0$, we may further assume (by suitably changing $t_j$) that  $$|I_{\omega_{t_j}}(\phi_{t_j,\beta}^{\epsilon_j})-I_{\omega_{t_0}}(\phi_{t_0,\beta}^{\epsilon_j})|=1.$$ 

The quantity $I_{\omega_{t_0}}(\phi_{t_0,\beta}^\epsilon)$ is uniformly bounded for $\epsilon\in(0, 1]$ by applying Ko\l odziej's estimate \cite{Ko} on the fixed manifold $X_{t_0}$, see \cite{CDS1}.  So we get a uniform bound on $||\phi_{t_j,\beta}^{\epsilon_j}||_{L^\infty}$ by Lemma \ref{lem2-12}, and then by Lemma \ref{lem2-12} we get a subsequence which converges to a K\"ahler-Einstein metric $\tilde\omega_{t_0,\beta}$ on $X_{t_0}\backslash D_{t_0}$, and moreover $\tilde\omega_{t_0,\beta}$ has a locally continuous K\"ahler potential $\tilde\phi_{t_0,\beta}$ and satisfies the Monge-Amp$\grave{\text{e}}$re equation: 

$$\tilde\omega_{t_0,\beta}^n=e^{-r(\beta)\tilde\phi_{t_0,\beta}}|S_{t_0}|_{h_{t_0}}^{2\beta-2}\vol_{h_{t_0}}$$
on $X_{t_0}\backslash D_{t_0}$.  It is thus a \emph{weak conical K\"ahler-Einstein metric} and must concide with $\omega_{t_0,\beta}$ by the uniqueness \cite{BBEGZ} and Remark \ref{rmk2-7}. This is a contradiction since the $I$ functional is continuous under the limit, i.e.
$$|I_{\omega_{t_0}}(\tilde\phi_{t_0,\beta})-I_{\omega_{t_0}}(\phi_{t_0,\beta})|=1.$$ 
\end{proof}

\begin{prop} \label{prop3-14}
Fix $r>0$. As functions of $t$ on $\p \Delta_r$,  the family $F_{\omega_t, (1-\beta)D_t}(\Psi^\epsilon(t, \cdot))$ converges to $F_{t,\beta}$ uniformly. In particular, $F_{t,\beta}$ is a continuous function of $t$. 
\end{prop}

\begin{proof}
\noindent By Lemma \ref{lem2-14}, there exists $C, \epsilon_0$ such that:

$$\sup_{\epsilon\in (0,\epsilon_0]}\sup_{\partial \X_r}||\Psi^\epsilon||_{L^\infty}<C$$
therefore we get a priori \emph{rough $C^2$ estimate} by a applying Chern-Lu inequality (see \cite{CDS1} for more detail):
$$ C^{-1}\omega_t\leq \omega_{{t,\beta}}^\epsilon\leq C\frac{\omega_t}{(|S_t|_{h_t}^2+\epsilon)^{1-\beta}}$$

\noindent and for any fixed $\delta>0$, 

$$\Psi^\epsilon\xrightarrow{C^\infty((\partial \X_r)\backslash \D_\delta)} \Psi.$$
Recall the definition of the log-Ding-functional:
\begin{equation*}
F_{\omega_t, (1-\beta)D_t}(\phi)
=-\sum_{i=0}^n\frac{1}{(n+1)!} \int_{X_t}\phi (\omega_t+\sqrt{-1}\partial\bar\partial \phi)^i\wedge\omega_t^{n-i}
-\frac{V}{r(\beta)}\log \frac{1}{V}\int_{X_t} e^{-r(\beta)\phi}|{S}_t|_{h_t}^{2\beta-2}\text{vol}_{h_t}
\end{equation*}

\noindent To deal with the first term in the log-Ding-functional, around the divisor $\D$, the above bound together with Equation (\ref{eqn2-4}) gives 
$$(\omega_{t, \beta}^\epsilon)^i\omega_t^{n-i}\leq C(\omega_{t, \beta}^\epsilon)^n\leq C\frac{\omega_t^n}{(|S_t|_{h_t}^2+\epsilon)^{1-\beta}}$$
therefore 
$$|\int_{X_t\cap\D_\delta} \phi^\epsilon_{t,\beta}(\omega_{t, \beta}^\epsilon)^i\omega_t^{n-i}|
\leq C^{n-i}||\phi_{t,\beta}^\epsilon||_{L^\infty}\int_{X_t\cap\D_\delta}\frac{\omega_t^n}{(|S_t|_{h_t}^2+\epsilon)^{1-\beta}}<C(\delta)$$ 
with $C(\delta)\to 0$ as $\delta$ goes to $0$. For any fixed $\delta>0$ small, the integrals on the complement $X_t\backslash \D_\delta$ converges by the smooth convergence of the potentials , therefore $F_{\omega_t}^0(\Psi^\epsilon(t,\cdot))$ converges to $F_{\omega_t}^0(\phi_{t,\beta})$ uniformly. Similarly the second term in the log-Ding-functional converges uniformly.
\end{proof}

\begin{prop}\cite{PS} \label{prop3-15} The Dirichlet problem (\ref{eqn2-3}) with $\Psi=\Psi^\epsilon$ has a generalized solution $\Phi^\epsilon$ which is  uniformly bounded on $\mathcal X$ (i.e. $||\Phi^\epsilon||_{L^\infty(\mathcal{X})}<C$ for a $C$ independent of $\epsilon$), and locally $C^{1, \alpha}$ away from the singular points of $\mathcal X$. 
\end{prop}

\begin{proof}
Take any log resolution, denoted by $\pi: \tilde{\mathcal{X}}\rightarrow \mathcal{X}$. Since $\mathcal{X}$ sits inside projective space, $\tilde{\mathcal{X}}$ also could be chosen to sit inside some projective space, with a Fubini-Study metric $\Omega_{FS}$. Then $\pi^*\Omega+\delta\Omega_{FS}$ is a family of smooth K\"ahler metrics in $\tilde{\mathcal{X}}$ for $\delta\in (0, 1)$. Try to solve the family of equations corresponding to $\delta\in (0,1)$:

 \begin{equation*} \label{}  \left\{ \begin{array}{ll}
               (\pi^*\Omega+\delta \Omega_{FS}+\sqrt{-1}\partial\bar\partial \Phi)^{n+1}=0;\\
               \Omega+\delta\Omega_{FS}+\sqrt{-1}\partial\bar\partial \Phi\geq 0\\
               \Phi|_{\partial \mathcal{X}_r}=\Psi.
               \end{array}\right.
               \end{equation*}
This is solvable for each $\delta$ with the solution $\underline{\Phi_\delta}$ satisfying a priori bound:

$$||\underline{\Phi_\delta}||_{L^\infty}<C$$
$$|\sqrt{-1}\p\bp\underline{\Phi_\delta}|<C$$
for any $K\subset \subset \tilde{\mathcal{X}}\backslash \cup E_i$. By letting $\delta$ go to $0$, we achieved the generalized solution for Equation \ref{eqn2-3} with the required regularity and estimate claimed.
\end{proof}

 Denote by $\Phi^\epsilon_t$ the restriction of $\Phi^\epsilon$ on  $X_t$. For $t\in \Delta_r$,   let 
$$f^\epsilon(t)=F_{\omega_t}^0(\Phi^\epsilon_t)$$
and 
$$g^\epsilon(t)=F_{\omega_t}^1(\Phi^\epsilon_t)$$
Then by definition the log-Ding-functional $F_{\omega_t, (1-\beta)D_t}(\Phi^\epsilon_t)$ is the sum of these two functions. 

By the general theory on \emph{positivity of direct image bundle} according to \cite{Bern} and positivity of \emph{Deligne Pairing}  according to \cite{Zhang}, both pieces of the log-Ding-functional are subharmonic functions of $t\in \Delta^*$. Now we suppress the superscript $\epsilon$ for simplicity and prove the continuity and subharmonicity of $g$ and $f$ on $\Delta_r$ respectively.

\begin{prop}\label{prop2-17}
$g$ is continuous and subharmonic on $\Delta_r$. 
\end{prop}
\begin{proof}

From the proof of Proposition \ref{prop3-14} $g(t)$ is continuous on $\Delta^*$.  Next we will prove $g$ is subharmonic on $\Delta_r^*$. First we notice it suffices to prove this for $t$ in a small disc $\Delta'$ in $\Delta^*$. We view $e^{-r(\beta)\Phi} |S_t|_{h_t}^{2\beta-2}\vol_{h_t}$ as a singular hermitian metric on $\mathcal L=K_{\mathcal X/\Delta'}^{-1}$ over $\pi^{-1}(\Delta')$ with non-negative curvature $r(\beta)(\lambda^{-1}\omega_{FS}+i\partial\bar\partial \Phi)+2\pi(1-\beta)[\mathcal{D}]$.  By well-known approximation theorem for plurisubharmic functions on K\"ahler manifolds (for example, Proposition 2.1 (2) in \cite{Bern}), for any fixed $\epsilon$, we can find a sequence of smooth hermitian metrics, written as  $e^{-\Psi_i}\vol_{h_t}$ where $\Psi_i$ is a function on $\pi^{-1}(\Delta')$ that decreasingly converges to $r(\beta)\Phi+(1-\beta)\log |S_t|^2_{h_t}$, and with $\Omega+\sqrt{-1}\p\bp\Psi_i\geq 0$. Denote by $\Psi_{i, t}$ the restriction of $\Psi_i$ on $X_t$. Now consider the direct image bundle $E$ with fibers $E_t=\Gamma(X_t, K_{\mathcal X/\Delta'}^{-1}\otimes K_{X_t})$, the canonical section $\textbf{1}$ has $L^2$ norm given by 
$$||\textbf{1}||^2_t=\int_{X_t}e^{-\Psi_{i, t}}\vol_{h_t}. $$
By Berndtsson's positivity of the direct image bundle \cite{Ber2}, $g_i(t)=-\log \int_{X_t}e^{-\Psi_{i, t}}\vol_{h_t}$ is a (smooth) subharmonic function over $\Delta'$. By construction $g_i(t)$ decreasingly converges to $g(t)$ on $\Delta'$, so it follows that $g(t)$ is also subharmonic. 

By the calculation of \cite{Li}, $g$ is continuous at $t=0$, therefore $g$ is subharmonic on the whole disk $\Delta_r$.
\end{proof}

\begin{prop} \label{prop2-18}
$f$ is continuous on $\Delta_r$. 
\end{prop}

\begin{proof}
We only prove the continuity at $t=0$, and the case for $t\neq 0$ is easier. By \cite{BT2} for any $i=0, \cdots, n$, the measure $\mu_0=(\omega_0+i\p\bp\Phi_0)^i\omega_0^{n-i}$ is a regular non-pluripolar Borel measure for bounded PSH function. In particular, $\int_{X_0^{sing}}(\omega_0+i\p\bp\Phi_0)^i\omega_0^{n-i}=0$. So for any $\delta>0$, we may choose $E^\delta_0$ to be the complement of a small neighborhood  $X_0^{sing}$ so that for all $i=0, \cdots, n$
\begin{equation}\label{eqn3-5}
\int_{X_0}(\omega_0+i\p\bp \Phi_0)^i\omega_0^{n-i}-\int_{E^\delta_0}(\omega_0+i\p\bp \Phi_0)^i\omega_0^{n-i}\leq\delta,
\end{equation}
and the boundary $\partial E_0^\delta$ does not have mass under the measures $\mu_0=(\omega_0+i\p\bp\Phi_0)^i\omega_0^{n-i}$ . We then extend $E^\delta_0$ to a smooth family of open subsets $E^\delta_t$ in $X_t$. 
We claim 
\begin{equation} \label{eqn3-3}
\lim_{t\rightarrow0} \int_{E^\delta_t} \Phi_t (\omega_t+i\p\bp\Phi_t)^i\omega_t^{n-i}=\int_{E^\delta_0}\Phi_0 (\omega_0+i\p\bp\Phi_0)^i\omega_0^{n-i}\end{equation}
and
\begin{equation} \label{eqn3-4}
\lim_{t\rightarrow0} \int_{E^\delta_t}  (\omega_t+i\p\bp\Phi_t)^i\omega_t^{n-i}=\int_{E^\delta_0} (\omega_0+i\p\bp\Phi_0)^i\omega_0^{n-i}. 
\end{equation}
Given the claim for the moment, we finish the proof. Since 
\begin{eqnarray*}
&&|\int_{X_t\setminus E^\delta_t} \Phi_t (\omega_t+i\p\bp\Phi_t)^i\omega_t^{n-i}|\\
&\leq& |\Phi|_{L^\infty}\int_{X_t\setminus E^\delta_t}  (\omega_t+i\p\bp\Phi_t)^i\omega_t^{n-i}\\
&=& |\Phi|_{L^\infty} (\int_{X_t}(\omega_t+i\p\bp\Phi_t)^i\omega_t^{n-i}-\int_{E^\delta_t}  (\omega_t+i\p\bp\Phi_t)^i\omega_t^{n-i}), 
\end{eqnarray*}
we have 
$$\limsup_{t\rightarrow0} |\int_{X_t\setminus E^\delta_t} \Phi_t (\omega_t+i\p\bp\Phi_t)^i\omega_t^{n-i}|\leq 
 |\Phi|_{L^\infty}  \delta$$
It follows from (\ref{eqn3-5}), (\ref{eqn3-3}), (\ref{eqn3-4}) that 
$$\limsup_{t\rightarrow0} |f_t-f_0|\leq |\Phi|_{L^\infty}\delta+\delta$$
Let $\delta\rightarrow0$ we obtain the desired continuity.

Now we prove the claim.
By choosing a finite open cover of $E^\delta_t$ it suffices to prove these two convergence properties for $E^\delta_t$ replaced by a continuously varying family $U_t$ of open sets in $\mathbb{C}^n$ where the open sets $U_0$ could be arranged to have zero mass under the measures $(\omega_0+i\p\bp\Phi_0)^i\omega_0^{n-i}$. By \cite{BT2},  $\mu_t=(\omega_t+i\p\bp\Phi_t)^i\omega_t^{n-i}$ converges to $\mu_0=(\omega_0+i\p\bp\Phi_0)^i\omega_0^{n-i}$ as currents. In this situation, the measure $\mu_t(U_t)$ would also converge to $\mu_0(U_0)$ since the boundary does not carry mass. This proves Equation (\ref{eqn3-4}) and Equation (\ref{eqn3-3}) follows from similar argument.
\end{proof}

\begin{prop}\label{prop2-19}
$f$ is subharmonic on $\Delta_r$.
\end{prop}

\begin{proof}
The strategy of showing the subharmonicity of $f$ is similar to that of $g$ in Proposition \ref{prop2-17}, on $\pi^{-1}(\Delta')$ for each small disc $\Delta'\subset \Delta^*$ and fixed $\epsilon$, we may approximate $\Phi$ by a decreasing sequence of smooth functions $\Phi_i$ with $\Omega+\sqrt{-1}\p\bp\Phi_i>0$. Denote by $\Phi_{i, t}$ the restriction of $\Phi_i$ on $X_t$. By the continuity of $\Phi$ and Dini's theorem,  as $i\rightarrow\infty$, $\Phi_{i,t}$ converges  to $\Phi_t$ uniformly on $\pi^{-1}(\Delta')$. It then follows by an easy integration by part argument that $f_i(t)=F_{\omega_t}^0(\Phi_{i,t})$ converges uniformly to $f(t)$ on $\Delta'$. Then the subharmonicity of $f$ would then follow from the subharmonicity of $F_{\omega_t}^0(\Phi_{i, t})$. The latter is well-known, and we recall briefly. The  line bundle $\mathcal{L}=K_{\mathcal X/\Delta^*}^{-1}$ over $\mathcal{X}$ gives rise to a line bundle  $L=<\mathcal{L}, \cdots, \mathcal{L}>$ over $\Delta^*$, called the \emph{Deligne pairing} \cite{De}, \cite{Zhang},\cite{Ber2}, depending multi-linearly on the $n+1$ components.  If $e^{-\psi}$ is a Hermitian metric on $\mathcal{L}$, then there is a natural Hermitian metric $e^{-\psi_D}$ on $L$ with curvature:
 
 $$\sqrt{-1}\partial\bar\partial \psi_D=\pi_*(\sqrt{-1}\partial\bar\partial \psi)^{n+1}=\int_{X_t} (\sqrt{-1}\partial\bar\partial \psi)^{n+1}$$
 
  If $e^{-\psi'}=e^{-\psi-\phi}$ is another Hermitian metric on $\mathcal{L}$, then the \emph{change of metric formula} on $L$ is given by:
 
 $$\psi_D'-\psi_D=\frac{1}{(n+1)!}\sum_{k=0}^n \int_{X_t} \phi(\sqrt{-1}\partial\bar\partial \psi')^{n-k}\wedge(\sqrt{-1}\partial\bar\partial \psi)^k$$
 
\noindent  In our setting, 
 $$F_{\omega_t}^0(\Phi^{\epsilon}_{i, t})
 =-\frac{1}{(n+1)!}\sum_{k=0}^n \int_{X_t} \Phi_{i, t}(\omega_t
 +\sqrt{-1}\partial\bar\partial \Phi_{i, t})^{n-k}\wedge \omega_t^k$$
 
\noindent precisely gives the ``change of metric'' on the \emph{Deligne Pairing} $L$ resulted from changing the Hermitian metrics from $h_{\Omega}e^{-\Phi_i}$ to $h_\Omega$ on $\mathcal{L}$, where $h_\Omega$ is the hermitian metric with curvature $\Omega$.
Therefore,

$$\sqrt{-1}\p\bp f_i=\int_{ X_t}\Omega^{n+1}- (\Omega+\sqrt{-1}\p\bp\Phi_i)^{n+1}$$

\noindent and for any smooth nonnegative function $\chi$ supported on $\Delta'$,

\begin{align*}
\int_{\Delta'} f\sqrt{-1}\p\bp \chi &=\lim_{i\to \infty}\int_{\Delta'} f_i\sqrt{-1}\p\bp \chi \\
&=\lim_{i\to \infty} \int_{\Delta'} \chi \int_{ X_t} -(\Omega+\sqrt{-1}\p\bp\Phi_i)^{n+1}+\Omega^{n+1}\\
&=\lim_{i\to \infty}\int_{\pi^{-1}(\Delta')}\pi^*\chi \cdot\{\Omega^{n+1}-(\Omega+i\p\bp \Phi^{\epsilon}_{i})^{n+1}\}\\
&=\int_{\pi^{-1}(\Delta')} \pi^*\chi \cdot \Omega^{n+1}\geq 0
\end{align*}

\noindent where in the last equality we use the fact that $(\Omega+\sqrt{-1}\p\bp \Phi_{i})^{n+1}$ converges to $(\Omega+\sqrt{-1}\p\bp \Phi)^{n+1}$ as a $(n+1,n+1)$ current on $\pi^{-1}(\Delta')$ by the monotone convergence theorem (see Theorem 2.1, \cite{BT1}). Therefore $f$ is subharmonic on $\Delta_r$ since it is continuous at $t=0$.
\end{proof}

%%%%%%%%%%%%%%%%%%%%%%
From Propositions \ref{prop2-17}, \ref{prop2-18}, \ref{prop2-19} we see that $F_{\omega_t}(\Phi_t^\epsilon)$ is a continuous subharmonic function on $\Delta_r$, so by \emph{maximum principle} 
$$\sup_{\Delta_r} F_{\omega_t,(1-\beta)D_t}(\Phi_t^\epsilon)\geq F_{\omega_0,(1-\beta)D_0}(\Phi_0^\epsilon)\geq F_{0,\beta}. $$
Let $\epsilon\rightarrow0$, using Proposition \ref{prop3-14} we get 
$$\sup_{\Delta_r} F_{t,\beta}\geq F_{0,\beta}. $$
Now let $r\rightarrow 0$, we get 
$$\limsup_{t\rightarrow0} F_{t,\beta}\geq F_{0,\beta}. $$ 
This proves the statement about the log-Ding-functional in Proposition  \ref{prop2-10}.

%%%%%%%%%%%%%%%%%%%%%%

\subsubsection{Lower bound of log-Mabuchi-functional}

Now we prove the statement about log-Mabuchi-functional in Proposition \ref{prop2-10}. By Proposition \ref{logDM},  it suffices to prove the following

\begin{prop} \label{prop2-20}
The function $\int_{X_t} (h_{\omega_t}-\log |S_t|_{h_t}^{2-2\beta})\omega_t^n$ is uniformly bounded below as  $t\rightarrow 0$.
\end{prop}
First of all, it is easy to see that $|S_t|_{h_t}^2$ is uniformly bounded above, so $-\int_{X_t}\log |S_t|_{h_t}^{2-2\beta} \omega_t^n$ is uniformly bounded below. It suffices to show $\int_{X_t}h_{\omega_t}\omega_t^n$ is uniformly bounded below. 
So the main issue is to control the behavior of the Ricci potential $h_t$  as $t\rightarrow0$.  For this we follow \cite{DT} and \cite{Li}, and use resolution of singularities.  
Recall from before that we have fixed an embedding of $\iota: \mathcal {X}\to \mathbb P^N\times \C$ using the sections of $K_{\mathcal X/\C}^{-\lambda}$, and we have chosen smooth hermitian metric $h$ on $\O(1)$ over $\P^N\times \C$ which induces hermitian metrics $h_{\Omega}$ on $K_{\mathcal X/\C}^{-1}$ with curvature form $\lambda^{-1}i\omega_{FS}+idt\wedge d\bar t$. Restricting to each fiber $X_t$, $h_\Omega$ defines a volume form $\vol_{h_t}$ which is smooth on the smooth part of $X_t$ and varies smoothly. To be explicit, we choose a local generator  $v$ of $\mathcal O (\lambda K_{\mathcal X/\C})$ in a neighborhood of a point, then we have 
\begin{equation} \label{eqn3-6}
\vol_{h}=\sqrt{-1}^{n^2}|v_t^*|_{h_\Omega}^{2/\lambda}(v_t\wedge\bar v_t)^{1/\lambda}.
\end{equation}
 Note $|v_t^*|$ is  a local smooth non vanishing function.  By definition, we have $h_t=\log(\vol_{h_t}/\omega_t^n)$, where $\omega_t=\lambda^{-1}\iota^*\omega_{FS}$. 
\noindent Take any log resolution $\mu: \tilde{\mathcal{X}}\to \mathcal{X}$, then we have

\begin{equation*}
K_{\tilde{\mathcal{X}}/\mathbb{C}}+\mathcal{X}_0'=\mu^*(K_{\mathcal{X}/\mathbb{C}}+\mathcal{X}_0)-\sum_i b_iE_i
\end{equation*}

\noindent where $E_i's$ are exceptional divisors and $b_i$ is the log discrepancy of $E_i$, and $\mathcal{X}_0'\cup \cup_i E_i$ have simple normal crossings.
By the adjunction Formula

\begin{equation*}
K_{\mathcal{X}_0'}=(\mu|_{\mathcal{X}_0'})^*K_{\mathcal{X}_0}-\sum_i b_i E_i|_{\mathcal{X}_0'}
\end{equation*}

\noindent Since $\mathcal{X}_0$ is KLT, $b_i<1$ (This also follows from the well-known \emph{Inversion of Adjunction} \cite{K}).\\

Suppose 
\begin{equation} \label{eqn3-7}
\mathcal{X}_0'=\mu^*\mathcal{X}_0-\sum_i a_i E_i
\end{equation}

\noindent for some positive integers $a_i$, then we have
\begin{equation} \label{eqn3-8}
 K_{\tilde{\mathcal{X}}/\mathbb{C}}=\mu^*K_{{\mathcal{X}/\mathbb{C}}}+\sum_i (a_i-b_i) E_i.
\end{equation}
From this equation and the definition of relative canonical line bundle one can calculate, as in \cite{Li}, the behavior of $\mu^*\vol_{h_t}$ in a neighborhood of a point $\tilde x\in \mathcal X_0'\cup\bigcup_i E_i$. There are three cases
\begin{enumerate}
\item $\tilde x\in  \mathcal{X}_0'$ but not on any of the $E_i$'s. In this case as $t\rightarrow 0$, $\mu^*\vol_{h_t}$ converges smoothly to the volume form $\mu^*\vol_{h_0}$ in a neighborhood of $x$. 
\item $\tilde x$ is in the intersection of $\mathcal X_0'$ with exactly $m$ exceptional divisors, say $E_1, \cdots E_m$. Then clearly $m\leq n$. We may choose local holomorphic coordinates $w_0, \cdots, w_n$ around $\tilde x$ so that $\mathcal X_0'$ is defined by $w_0=0$ and $E_i$ is defined by $w_i=0$ for $i=1, \cdots, m$. (\ref{eqn3-7}) means that we can assume $t=w_0\Pi_{i=1}^m w_i^{a_i}$.  Then by (\ref{eqn3-6}), (\ref{eqn3-8}) we have 
\begin{eqnarray*}
\mu^*\vol_{h_t}&=& P(w)|g(w)|^2 \iota_{\p_t\otimes \bar \p_t}(\Pi_{i=1}^m |w_i|^{2a_i-2b_i} dw_0\wedge d\bar w_0 \cdots dw_n\wedge d\bar w_n )\\
&=&
P(w)|g(w)|^2\Pi_{i=1}^m(|w_i|^{-2b_i}dw_i\wedge d\bar{w}_i)
\wedge \Pi_{j=m+1}^n (dw_j\wedge d\bar{w}_j), 
\end{eqnarray*}
where $P$ is a non-zero smooth function and $g$ is a non-vanishing holomorphic function. 
\item $\tilde x\notin \mathcal X_0'$ but $\tilde x$ is in the intersection of exactly $m$ exceptional divisors, say $E_1, \cdots, E_m$. We choose local holomorphic coordinates $w_1, \cdots, w_{n+1}$ around $\tilde x$ so that $E_i$ is defined by $w_i=0$ for $i=1, \cdots, m$ and $t=\Pi_{i=1}^m w_i^{a_i}$. 
\begin{eqnarray*}
\mu^*\vol_{h_t}&=&P(w)|g(w)|^2 \iota_{\p_t\otimes \bar \p_t}(\Pi_{i=1}^m |w_i|^{2a_i-2b_i} dw_1\wedge d\bar w_1 \cdots dw_{n+1}\wedge d\bar w_{n+1} )\\
\\&=& P(w)|g(w)|^2 \Pi_{i=1}^m|w_i|^{2\beta_i} dw_2\wedge d \bar w_2\cdots\wedge dw_{n+1}\wedge d\bar { w}_{n+1}, 
\end{eqnarray*}
where $P$ and $g$ are as before, and $\beta_i$ are numbers that can be calculated explicitly in terms of $a_i$ and $b_i$ (see \cite{Li}), but we do not need this for our purpose here. 
\end{enumerate}

We also have a good understanding of the behavior of  $\mu^*\omega_t^n$. Let $\tilde\Omega$ be a smooth K\"ahler metric on $\mathcal X$, and write $\mu^*\omega_{FS}^n=Q\tilde\Omega$.

\begin{lem}
$Q$ is a smooth function that vanishes on $\cup _{i}E_i$. In particular near each point $\tilde x \in \cap_{i=1}^m E_i$, we have $Q(w_1,\cdots, w_{n+1})=O(|w_1\cdots w_m|)$, where $w_1, \cdots, w_{n+1}$ are local holomorphic coordinates and $E_i=\{w_i=0\}, i=1, \cdots, m$ are the first $m$ coordinate planes.
\end{lem}
\begin{proof}
This follows from  the fact that $E_i$'s are exceptional divisors that are mapped to a subvariety of $X_0\subset \P^N$ of dimension at most $n-2$. 
\end{proof}

Now we write 
\begin{eqnarray*}
\int_{X_t} h_{\omega_t}\omega_t^n&=&
\int_{X_t}\log \frac{\vol_{h_t}}{\omega_t^n}\omega_t^n\\
&=&
\int_{\mathcal X_t}Q\log \frac{ \mu^*\vol_{h_t}}{\tilde \Omega}\tilde\Omega-\int_{\mathcal X_t}Q \log Q\tilde\Omega
\end{eqnarray*}

The second is uniformly bounded since $Q$ is smooth, and for the first integral, we notice that the integrand is uniformly bounded by the above discussion. In sum this proves Proposition \ref{prop2-20}.

\subsection{Finish of the proof of Theorem \ref{thm1-3}}

By  Corollary \ref{cor3-9} and Proposition \ref{prop2-10} we can find $C>0$, and a sequence $t_r\rightarrow0$, so that the log-Mabuchi-functional $M_{\omega_{t_r}, (1-\underline{\beta})D_{t_r}}(\phi)\geq -C$ for $\phi\in  C^{1,1}(X_{t_r})$. \\
 
 We also fix a $\hat\beta \in (0, 1-\lambda^{-1})$, let $V_t=\int_{X_t} |S_t|_{h_t}^{2\beta-2}\text{vol}_{h_t}$, Proposition \ref{prop2-19} shows that $V_t$ is a continuous function of $t\in \Delta_r$. Together with Proposition \ref{prop2-20} and the elementary inequality $x\log x\geq -e^{-1}$,  for any $\phi\in  C^{1,1}(X)$ and $t\in \Delta$, we have

\begin{align*}
M_{\omega_t, (1-\hat\beta)D_t}(\phi)
=&\int_{X_t} \log \frac{\omega_\phi^n}{e^{H_{\omega_t, (1-\beta)D_t}}\omega_t^n}\frac{\omega_\phi^n}{n!}-r(\hat\beta)(I-J)_{\omega_t}(\phi)+\int_{X_t} H_{\omega_t, (1-\beta)D_t}\frac{\omega_t^n}{n!}\\
&\quad\quad+\int_{X_t}H_{\omega_t, (1-\beta)D_t}\frac{\omega_t^n}{n!}\\
\geq& -\int_{X_t}e^{-1}e^{H_{\omega_t, (1-\beta)D_t}}\frac{\omega_t^n}{n!}-r(\hat\beta)(I-J)_{\omega_t}(\phi)+\int_{X_t}H_{\omega_t, (1-\beta)D_t}\frac{\omega_t^n}{n!}\\
\geq& -r(\hat\beta)(I-J)_{\omega_t}(\phi)-e^{-1}V_t-C\\
\geq& -r(\hat\beta)(I-J)_{\omega_t}(\phi)-C
\end{align*}

As observed in \cite{LS}, the log-Mabuchi-functional is linear in $\beta$, so by Corollary \ref{cor3-9} and Proposition \ref{prop2-10} there exist a constant $C>0$ independent of $r$ such that for all $\beta\in [\hat\beta, \underline{\beta})$,  we have
  
 \begin{equation} \label{proper}
{M}_{\omega_{t_r}, (1-\beta)D_{t_r}}(\phi)
 \geq \delta_\beta(I-J)_{\omega_{t_r}}(\phi)-C, 
 \end{equation}    
 where $\delta_\beta=-r(\hat\beta)\frac{\underline{\beta}-\beta}{\underline{\beta}-\hat\beta}$. 
 Since the K\"ahler-Einstein metric $\omega_{t,\beta}$ achieves the minimum of the log-Mabuchi-functional \cite{BBEGZ}, we have
$$0={M}_{\omega_{t_r}, (1-\beta)D_{t_r}}(\omega_{t_r})\geq {M}_{\omega_{t_r}, (1-\beta)D_{t_r}}(\omega_{t_r,\beta})\geq \delta_\beta(I-J)_{\omega_{t_r}}(\omega_{t_r,\beta})-C$$
So 

\begin{equation}\label{eqn2-12}
I_{\omega_{t_r}}(\phi_{t,\beta})\leq (n+1)(I-J)_{\omega_{t_r}}(\omega_{t_r,\beta}) \leq C\delta_\beta^{-1}
\end{equation}

Now we fix $\epsilon>0$ small,  then for any $\beta\in [1-\lambda^{-1}+\epsilon, \underline{\beta}-\epsilon]$, by \cite{CDS1} for any $t\in \Delta^*$ we could approximate $\omega_{t_r, \beta}$ by smooth K\"ahler metrics $\omega_{t,\beta}^\epsilon$ which satisfies Equation \ref{eqn2-4}.  Then by Lemma \ref{lem2-12}:

$$||\phi_{t_r,\beta}||_{L^\infty}\leq C_\epsilon$$
for constant $C_\epsilon>0$ independent of $r\in (0,1]$ (however may not be uniformly bounded when $\epsilon\rightarrow 0$).

To finish the proof of Theorem \ref{thm1-3}, we need to pass from the sequence $t_r, r\in (0,1]$ to all $t \in \Delta^*$. Lemma \ref{lem2-12} implies that we could take a subsequential limit $\tilde\phi_{\infty,\beta}$ as $r\to 0$ which satisfies:

$$(\omega_0+\sqrt{-1}\partial\bar\partial \tilde\phi_{\infty, \beta})^n=e^{-r(\beta)\tilde\phi_{\infty, \beta}}|S_0|_{h_0}^{2\beta-2}\vol_{h_0}.$$

\noindent This equation implies that $\tilde\omega_{\infty, \beta}=\omega_0+\sqrt{-1}\partial \bar\partial \tilde\phi_{\infty, \beta}$ is a \emph{weak conical K\"ahler-Einstein metric} on the log Fano pair $(X_0,(1-\beta)D_0)$. Since $Aut(X_0, D_0)$ is discrete, the uniqueness theorem of \cite{BBEGZ} implies that $\tilde\omega_{\infty,\beta}$ is the unique \emph{weak conical K\"ahler-Einstein metric} $\omega_{t,\beta}$ on $(X_0, (1-\beta)D_0)$, and the whole family $\omega_{t_r,\beta}$  converge to it in the same sense. Then Theorem \ref{thm1-3} is finished by the next proposition, with the $\underline{\beta}$ replaced by $\underline{\beta}-\epsilon$ in the following proposition. 

\begin{prop}\label{prop3-24}
 Under the same assumption as Proposition \ref{prop2-10}, for any $\beta\in [1-\lambda^{-1}+\epsilon, \underline{\beta}-\epsilon]$,
  $$\limsup_{\delta\to 0} \max_{|t|=\delta} I_{\omega_t}( \omega_{t,\beta})<\infty.$$ In particular,
  $$||\phi_{t,\beta}||_{L^\infty}<C_\epsilon$$
 for a constant $C$ independent of $t\in \Delta^*$.
 \end{prop}
 
\begin{proof}
Argue by contradiction. Since we already have $I_{\omega_{t_r}}(\phi_{t_r,\beta})\leq C$ by the inequality \ref{eqn2-12}, suppose the conclusion in the proposition is not true. Then we could pick $\epsilon_j\to 0$ and $|s_j|=\epsilon_j$, such that $I_{\omega_{s_j}}(\omega_{s_j,\beta})=C+1$. The same reasoning (by Lemma \ref{lem2-12}) as above shows that $\omega_{s_j,\beta}$ converges to $\omega_{0,\beta}$. Clearly this is a contradiction since the $I$ functional converges under the limit:

$$\lim_{j\to \infty} I_{\omega_{s_j}}(\omega_{s_j,\beta})=C+1=I_{\omega_{0}}(\omega_{0,\beta})=\lim_{j\to \infty} I_{\omega_{t_{\epsilon_j}}}(\omega_{t_{\epsilon_j},\beta})\leq C$$
\end{proof}

\section{Gromov-Hausdorff convergence under $L^\infty$ bound on K\"ahler potentials}\label{section3}

In this section we study the behavior of (conical) KE metrics in a $\Q$-Gorenstein smoothing family, under the additional hypothesis of a uniform estimate on the KE K\"ahler potentials (with respect to some Fubini-Study background metric). In particular, we show that the Gromov-Hausdorff limit as $t\rightarrow 0$ is unique and equal to the weak (conical) KE metric on the central fiber.

As we explained, we can reduce to the following setting: we may assume that the $\Q$-Gorenstein smoothing $\pi:\mathcal{X}\rightarrow \Delta$ of a $\Q$-Fano variety is $\lambda$-plurianticanonically embedded in $\X\subset \C\P^{N(\lambda)}\times\C$.  Moreover assume we have a relative $\lambda$-plurianticanonical divisor $\mathcal{D}$ that is smooth over $\Delta^*=\Delta\setminus\{0\}$ and so that $(X_0, (1-\beta)D_0)$ is KLT for all $\beta\in (0, 1)$.

\begin{thm} \label{thm3-1}
Let $\pi:(\mathcal{X},\mathcal{D})\rightarrow \Delta$ be a $\Q$-Gorenstein smoothing as above. Let $\beta\in (1-\lambda^{-1}, 1]$, assume that for any $t\in \Delta$, there is a weak conical K\"ahler-Einstein metric $\omega_{t, \beta}$ on $(X_t, (1-\beta)D_t)$, which is genuinely conical for $t\neq 0$ and such that $\omega_{t, \beta}=\omega_{t, FS}+\sqrt{-1}\p\bp\phi_{t,\beta}$ with $|\phi_{t,\beta}|_{L^\infty}$ uniformly bounded. 
Then the conical KE metrics on the smooth fibers converge to the weak KE metric on the central fiber in the Gromov-Hausdorff topology. Moreover, we have that $|\nabla_{\omega_{t, \beta}}\phi_{t,\beta}|$ is uniformly bounded and all higher derivatives of $\phi_{t,\beta}$ is uniformly bounded away from the singular set.
\end{thm}

\begin{rmk}
It follows from the proof below and the theorem in \cite{DS} that the same conclusion is true if the cone angles of the conical KE metrics vary and stay bounded below by $1-\lambda^{-1}+\delta$ for some $\delta>0$ since they will have a uniform diameter upper bound.
\end{rmk} 
For the proof of the above theorem \ref{thm3-1}  we need a couple of lemmas.

\begin{lem} Let $\pi:(\mathcal{X},\mathcal{D})\rightarrow \Delta$ a family as before. Then there exists a possibly very large $k \in\N$ such that for a $k$-th Veronese re-embedding of the family in $\P^{N(\lambda k)}$ we can still assume (for a suitable rescaling) $\omega_{t, \beta}=\omega_{t, FS}+\sqrt{-1}\p\bp\phi_{t,\beta}$ with $|\phi_{t,\beta}|_{L^\infty}$ uniformly bounded. Moreover, all possible Gromov-Hausdorff limits of $(X_t,(1-\beta)D_t, \omega_{t,\beta})$
as $t\rightarrow 0$ can be realized as weak conical KE KLT pairs plurianticanonically embedded in the same $\P^{N(\lambda k)}$. 
\end{lem}
\begin{proof}
By \cite{DS} and \cite{CDS2},  there is an integer $k>0$  so that all the Gromov-Hausdorff limits of the family $(X_t, (1-\beta)D_t, \omega_{t, \beta})$ as $t\rightarrow 0$ can be realized as limits of embeddings of $X_t$ into $\mathbb P^{N(\lambda k)}$ using $L^2$-orthonormal
sections of $K_{X_t}^{-\lambda k}$. It remains to show that the uniform estimate on the K\"ahler potential still holds under re-embeddings. But this follows immediately by noting that the difference of the two Fubini-Study metrics after the Veronese embeddings (and rescalings) is given by the $i\partial\bar\partial$ of a uniformly bounded function.
\end{proof}

Thus in the following we may assume $k=1$ in the above lemma. Denote by $H$ the chosen Hermitian metric on $\mathbb{C}^{N(\lambda)+1}$ given the back-ground Fubini-Study metric.  We also have a $L^2$ metric, denoted by $H_t$, on $H^0(X_t, K_{X_t}^{-\lambda})$ (which have been identifed once for all with $\C^{N(\lambda)+1}$) induced by the conical K\"ahler-Einstein metric $\omega_{t, \beta}$.  

\begin{lem} Under the previous hypothesis and notation, if the two euclidean norms $H_t$ and $H$ on $\C^{N(\lambda)+1}$ are uniformly equivalent then $(X_t,(1-\beta)D_t,\omega_{t,\beta})\rightarrow (X_0,(1-\beta)D_0,\omega_{0,\beta})$ in the Gromov-Hausdorff topology.  
\end{lem}
\begin{proof}
With a slight abuse of notation, let $X_t$ be the fiber in the original family seen inside the projective space $\mathbb P^{N(\lambda)}$.  Assume that the statement is not true. Then we can find a sequence of points $t_i\rightarrow 0$ so that the Gromov-Hausdorff metric distance, with respect to the conical KE,  between $X_{t_i}$ and $X_0$ stays bounded away from zero. Eventually taking a subsequence, we may  assume, by the previous lemma  that  for  the re-embedding given by  $L^2$-orthonormal sections $T(X_{t_i})\rightarrow W \subseteq \mathbb P^{N(\lambda)}$ as cycles as $t_i\rightarrow 0$. By the hypothesis on  uniform control of the norms, it follows that $T(X_{t_i})=g_i^\ast X_{ t_i}$ with $(g_i)\subseteq K \subseteq  PGL(\C^{N(\lambda)+1})$, where $K$ is compact. Thus, eventually taking a subsequence again, we may assume $g_i \rightarrow g_\infty\in K$ and $W=g_\infty^\ast X_0$ and similarly for the divisors. Combined with uniqueness of the KE metric, this is clearly a contradiction.  \end{proof}
 
\begin{rmk}
 Notice that it would not be true in general that the restriction of the Fubini-Study metrics on the fibers and the KE metrics  are uniformly  equivalent, as the case of orbifold singularities in dimension two clearly shows. 
\end{rmk}

We are now ready to give the proof of Theorem \ref{thm3-1}.
\begin{proof}[Proof of theorem \ref{thm3-1}]
 Note the $L^\infty$ bound on K\"ahler potentials is equivalent to that the induced Hermitian metrics $h_{t, \beta}$ and $h_{t, FS}$ on $K_{X_t}^{-\lambda}$ are uniformly equivalent.  Given any section $S=\sum_i a_i(t) S_i$, we have 
 \begin{eqnarray*}
||S||_{H_t}^2&=&\int_{X_t} |S|_{h_{t,\beta}}^2\omega_{t,\beta}^n\leq C\int_{X_t}|S|_{h_{t,FS}}^2\omega_{t,\beta}^n\\
&=&C\int_{X_t} \frac{|\sum_i a_i z_i|^2}{|z_0|^2+|z_1|^2+\cdots+|z_{N(\lambda)}|^2}{\omega}_{t,\beta}^n\\
&\leq& C \int_{X_t}\sum_i |a_i|^2\omega_{t,\beta}^n=CV\sum_i |a_i|^2=CV||S||_{H}^2.
\end{eqnarray*}
Since the volume $V$ is independent of $t$, we see $||S||_{H_t}$ is uniformly bounded above by $||S||_H$. 

To see the converse bound, we prove by contradiction. Suppose we have a sequence $t_\alpha\rightarrow 0$ and $S^\alpha=\sum_i a_i^\alpha(t_\alpha) S_i$ with $\sum_i |a_i^\alpha|^2=1$ but with $||S^\alpha||^2_{H_t}\rightarrow 0$. 
Using the $L^\infty$ bound, we obtain 
$$\int_{X_{t_\alpha}}  \frac{|\sum_i a_i^\alpha z_i|^2}{|z_0|^2+|z_1|^2+\cdots+|z_{N(\lambda)}|^2}\omega_{t_\alpha, \beta}^n \rightarrow 0.$$
Moreover, thanks to the $L^\infty$-bound hypothesis, we can use Lemma \ref{lem2-12}, and conclude that
$$\int_{X_{t_\alpha}}  \frac{|\sum_i a_i^\alpha z_i|^2}{|z_0|^2+|z_1|^2+\cdots+|z_{N(\lambda)}|^2}\omega_{t_\alpha, FS}^n \rightarrow 0.$$
By passing to a subsequence we may assume $a_i^\alpha\rightarrow a_i^\infty$ for all $i$, and 
$$\int_{X_0} \frac{|\sum_i a_i^\infty z_i|^2}{|z_0|^2+|z_1|^2+\cdots+|z_{N(\lambda)}|^2}\omega_{0, FS}^n=0.$$  
It follows that $\sum_i a_i^\infty S_i(0)=0$, which is a contradiction since $X_0$ is by hypothesis embedded by a complete linear system. This shows that $H$ is uniformly bounded above by $H_t$ and, using the previous lemma, it also finishes the proof of the first part of the theorem. 
 
The second part of the theorem follows from the fact \cite{DS} that under the $L^2$ embedding of $X_t$ into $\P^{N(\lambda)}$ we have $\omega_{t, \beta}=\omega_{t, FS}'+i\p\bp \psi_t$ with $|\nabla \psi_t|$ uniformly bounded and $\psi_t$ converges smoothly away from the singular set of the limit. Since $H_t$ is uniformly bounded, it is easy to see the same holds under the original embedding. 
\end{proof}

\section{Existence of K\"ahler-Einstein metrics} \label{section4}

Let $\D\in |-\lambda K_{\mathcal{X}/\Delta}| $ be as in the setting of Theorem \ref{thm1-3}, we define the following function
\begin{align*}
\beta_t:= \sup\{\beta\in (1-\lambda^{-1},1] \,|\, &\exists \mbox{ conical KE with cone angle } 2\pi\beta \mbox{ on } (X_t,D_t)\}, \mbox{for }t\neq 0;\\
\beta_0:= \sup\{\beta\in (1-\lambda^{-1},1] \,|\, &\exists \mbox{ weak conical KE with cone angle } 2\pi\kappa \mbox{ on } (X_0,D_0)\text{ for }\forall \kappa\leq\beta\}
\end{align*}

Notice by Theorem \ref{thm1-3}, we may assume $\beta_t\geq \underline{\beta}>1-\lambda^{-1}$ for all $t\in\Delta$. 
For $t\neq 0$, it follows from \cite{CDS0, CDS1, CDS2, CDS3, Ber2} that the existence of K\"ahler-Einstein metrics on $X_t$ with cone angle $2\pi\beta$ ($\beta\in (0, 1]$) along $D_t$ corresponds to the K-polystability of $(X_t, (1-\beta)D_t)$, and the latter condition satisfies an obvious interpolation property for $\beta$.  In particular, for $t\neq 0$, there  indeed exists  a K\"ahler-Einstein metric on $X_t$ with cone angle $2\pi \beta$ along $D_t$ for all $\beta\in (0, \beta_t)$. On the central fiber since it is one of our goal to establish the existence result, at this stage we do not have the interpolation property. This is the reason that the above definition we distinguish between the case $t\neq 0$ and $t=0$.

\subsection{Lower semi-continuity}
\begin{prop}\label{Semi-C}
$\beta_t$ is a lower semi-continuous function of $t\in \Delta$.
\end{prop}

\begin{proof}
The lower semicontinuity at $t\neq 0$ follows from Donaldson's implicit function theorem \cite{Do} for conical K\"ahler-Einstein metrics, since in our situation $Aut(X_t, D_t)$ is discrete for all $t\in \Delta$ (see Remark \ref{rmk2-7}). Suppose $\beta_t$ is not lower semi-continuous at $t=0$, i.e., $\liminf_{t\to 0}\beta_t=\beta_\infty< \beta_0\leq 1$.  Choose an increasing sequence $\beta_i<\beta_\infty$ with $\lim_{i\rightarrow\infty}\beta_i=\beta_\infty$. Given any $i$, for $t\neq 0$ with $|t|$ sufficiently small, by the above discussion  there exist genuinely conical K\"ahler-Einstein metrics $\omega_{t,\beta_i}$ on $(X_{t}, D_{t})$ with angle $2\pi\beta_i$. By the definition of $\beta_0$, we have a weak conical K\"ahler-Einstein pair $(X_0, D_0, \omega_{0,\beta}, \beta)$ for all $\beta\in (\underline{\beta}, \beta_\infty]$. So similar to the proof of Theorem \ref{thm1-3}, using Proposition \ref{prop2-10}, the linear interpolation property and Theorem \ref{thm3-1}, we see that $(X_0, (1-\beta_i)D_0, \omega_{0, \beta_i})$ is indeed the Gromov-Hausdorff limit of $(X_{t}, (1-\beta_i)D_{t}, \omega_{t,\beta_i})$ as $t\rightarrow0$. Using the diagonal arguments as in \cite{CDS2}, by passing to a subsequence we could take the Gromov-Hausdorff limit of the sequence $(X_0, (1-\beta_i)D_0, \omega_{0, \beta_i})$, and obtain a KLT weak conical K\"ahler-Einstein variety $(Y, (1-\beta_\infty)\Delta, \omega)$. However by Berman's theorem \cite{Ber2} we know $(X_0, (1-\beta_\infty)D_0)$ is K-polystable. So by \cite{CDS2} we see $(Y, \Delta)$ must be isomorphic to $(X_0, D_0)$, and by the uniqueness theorem \cite{BBEGZ} $(Y, (1-\beta_\infty)\Delta, \omega)$ is indeed isomorphic to $(X_0, (1-\beta_\infty)D_0, \omega_{0, \beta_\infty})$. 

Now let $\mathcal Z$ be the space of all $(X_t, (1-\beta)D_t, \omega_{t, \beta})$ with $0<|t|\leq 1/2$ and $\beta\in [\underline{\beta}, \beta_{t}]$. Let $\bar{ \mathcal Z}$ be the closure of $\mathcal Z$ under refined GH convergence (remembering the convergence of complex structures, as in \cite{DS}), and let  $\mathcal C$ be the space of $\bar{\mathcal Z}\setminus \mathcal Z$ consisting of limits with angle $2\pi\beta_\infty$. By \cite{CDS2}, \cite{CDS3} we can take the $L^2$ orthonormal embeddings to embed $(X_t, D_t)$ uniformly into a fixed projective space $\P^N$, and realize the refined GH convergence as the convergence in a fixed Chow variety $\mathbf{Ch}$ (of pairs of bounded degree), up to unitary transformations. In particular, we have an injective continuous map from $\bar{\mathcal Z}$ into $\mathbf{Ch}/U(N)$, which is a homeomorphism onto its image. We will henceforth view $\bar{\mathcal Z}$ (and hence $\mathcal C$) as a compact subspace of $\mathbf{Ch}/U(N)$. The advantage of this point of view is that the latter is naturally a metric space, as the quotient by a group of isometries of a compact metric space, and it makes sense to talk about a metric neighborhood. 

Now we observe by the above discussion that $(X_0, (1-\beta_\infty)D_0, \omega_{0,\beta_\infty})$ is in $\mathcal{C}$.

\begin{lem}\label{Uniqueness}
$\mathcal{C}=\{(X_0, (1-\beta_\infty)D_0, \omega_{0,\beta_\infty})\}.$
\end{lem}

\begin{proof}[Proof of the lemma \ref{Uniqueness}]
\noindent Choosing a small  open neighborhood $\mathcal{U}$ of $(X_0, (1-\beta_\infty)D_0, \omega_{0,\beta_\infty})$ in $\mathcal{C}$, we could assume all the elements in $\mathcal{U}$ have discrete automorphism group, since this condition is an open condition on $\mathbf{Ch}$. Now take any $(Y, (1-\beta_\infty)\Delta, \omega)$ in $\mathcal{U}$, which is the GH limit of some sequence $(X_{t_i}, (1-\beta_i)D_{t_i}, \omega_{t_i,\beta_i})$ with $\beta_i\rightarrow\beta_\infty$.  Write $\omega_{t_i, \beta_i}=\omega_{t_i}+i\p\bar\partial \phi_{t_i, \beta_i}$, where $\omega_{t_i}$ denotes the pull-back of the Fubini-Study metric on $X_i$ (normalized to have the correct cohomology class) under the $L^2$ holomorphic embedding with respect to the conical K\"ahler-Einstein metrics.  By \cite{CDS2}, we know $\phi_{t_i, \beta_i}$ is uniformly bounded in $L^\infty$.

 {\bfseries{Claim 1}}: There exists a $\delta>0$, so that for all $\beta\in (\beta_\infty-\delta, \beta_\infty)$ we can find $N$, such that  for all $i>N$, 
$$\sup_{\kappa\in[\beta, \beta^i]} \{I_{\omega_{t_i}}(\phi_{t_i, \kappa})-I_{\omega_{t_i}}(\phi_{t_i, \beta_i})\}\leq 1. $$

\emph{Proof of Claim 1}: Suppose this is not true. Then we could find a subsequence, still denoted by $t_i\to 0$ and $\kappa_i\to \beta_\infty$ such that $I_{\omega_{t_i}}(\phi_{t_i, \kappa_i})-I_{\omega_{t_i}}(\phi_{t_i, \beta_i})>1$. Since by Donaldson's implicit  function theorem \cite{Do} $\phi_{t_i, \kappa}$ depends continuously on $\kappa$ in the $L^\infty$ topology,  we know $\lim_{\kappa\rightarrow \beta_i} I_{\omega_{t_i}}(\phi_{t_i, \kappa})-I_{\omega_{t_i}}(\phi_{t_i, \beta_i})=0$. Therefore, we may by choosing a different sequence $\kappa_i$, assume that $I_{\omega_{t_i}}(\phi_{t_i, \kappa_i})-I_{\omega_{t_i}}(\phi_{t_i, \beta_i})=1$. The $L^\infty$ bound of $\phi_{t_i,\beta_i}$ implies the uniform bound of $I_{\omega_{t_i}}(\phi_{t_i,\beta_i})$ and therefore the uniform bound of $I_{\omega_{t_i}}(\phi_{t_i,\kappa_i})$. Then as before we achieve uniform $C^0$ bound on $\phi_{t_i, \kappa_i}$ and higher derivative bound away from the divisor and singularities of the central fiber.
This implies that
$(X_{t_i}, (1-\kappa_i)D_{t_i}, \omega_{t_i, \kappa_i})$ converges by subsequence to a weak conical K\"ahler-Einstein metric on the same pair $(Y, (1-\beta_\infty)\Delta, \omega)$. By  assumption  $\text{Aut}(Y, \Delta)$ is discrete, the uniqueness  of weak conical K\"ahler-Einstein metrics \cite{BBEGZ} implies that the limit of $\omega_{t_i,\kappa_i}$ must concide with the limit of $\omega_{t_i,\beta_i}$. On the other hand, as before we know the gap of the $I$ functional in the limit is still $1$, contradiction.

This claim guarantees that for some $\beta<\beta_\infty$, $||\phi_{t_i,\beta}||_{L^\infty}$ is bounded independent of $i$ such that $(X_{t_i}, (1-\beta)D_{t_i}, \omega_{t_i,\beta})$ converges to  $(Y, (1-\beta)\Delta, \omega)$ by Theorem \ref{thm3-1}. However previously it has already been proved that $(X_{t_i}, (1-\beta)D_{t_i}, \omega_{t_i, \beta})$ converges in GH sense to $(X_0, (1-\beta)D_0, \omega_{0,\beta})$ (also by Theorem \ref{thm3-1}). This concludes that $(Y, \Delta)=(X_0, D_0)$ and therefore $(Y, (1-\beta_\infty)\Delta, \omega)=(X_0, (1-\beta_\infty)D_0, \omega_{0,\beta_\infty})$. This proves that 
\begin{equation} \label{eqn4-1}
\mathcal{C}\cap \mathcal U=\{(X_0, (1-\beta_\infty)D_0, \omega_{0,\beta_\infty})\}.
\end{equation}

{\bfseries{Claim 2}}: $\mathcal{C}$ is connected.\\

\emph{Proof of Claim 2}: Define a family $C_\alpha=\cup_{0<|t|<\alpha}\{ (X_t, (1-\beta)D_t, \omega_{t,\beta})|\beta\in (\beta_t-\alpha,\beta_t)\}$ indexed by $\alpha\in (0,1)$ inside $\mathbf{Ch}$, which is pre-compact under the refined GH topology by \cite{CDS2}. Clearly $\lim_{\alpha\to o}C_\alpha=\mathcal{C}$ and each $C_\alpha$ is path-connected. The claim follows by applying the elementary point-set topology result Lemma \ref{connected} to this situation.
 
Now the Lemma \ref{Uniqueness} follows from Claim 2 and (\ref{eqn4-1}).
\end{proof}

\noindent \emph{Continuing Proof of Proposition \ref{Semi-C}:}

 By the definition of $\beta_t$ and \cite{CDS3},  for any fixed $i$, $(X_{t_i},(1-\beta)D_{t_i}, \omega_{t_i,\beta})$  converges by subsequence to some K\"ahler-Einstein pair $(W_{t_i}, (1-\beta_{t_i})\Delta_{t_i}, \omega_{t_i,\beta_{t_i}})$ with $Aut(W_{t_i}, \Delta_{t_i})$ containing a non-trivial one parameter subgroup. This limiting sequence would also converge by subsequence to the KLT pair $(X_0, (1-\beta_\infty)D_0,\omega_{0,\beta_\infty})$ by the above Lemma \ref{Uniqueness}. This is a contradiction since by assumption $Aut(X_0, D_0)$ is discrete.
  \end{proof}
 
Let us try to formulate some elementary point-set topology lemma that would implies the connectivity of $\mathcal{C}$ in this proof.
 
 \begin{lem}\label{connected}
Let $\mathcal{I}$ be a total order set with the least element ``o'', let $C_\alpha$ be a chain of subsets indexed by $\mathcal{I}\backslash o$ (i.e. $C_{\alpha_1}\subset C_{\alpha_2}$ if $\alpha_1<\alpha_2$) of a fixed compact metric space $A$ equipped with the subspace topology, denote $C_o=\lim_{\alpha\to o} C_\alpha:=\{\lim_{i\to \infty}x_{\alpha_i}| x_{\alpha_i}\in C_{\alpha_i},\text{ for some }\alpha_i\to o \}$ to be the limit set consisting of all possible sequential limits, if $C_\alpha$ is connected for any $\alpha\in \mathcal{I}\backslash o$ , then $C_o$ is also connected.
 \end{lem}
 
 \begin{proof}
 Suppose $C_o$ is not connected, which means $C_o=C_o^1\sqcup C_o^2$ for two closed nonempty subspaces of $C_o$. Since $A$ is compact, $C_o$ must also be compact and therefore $C_o^1, C_o^2$ are both compact. We then could find two disjoint open subsets $U, V\subset A$ (this implies $\overline U\cap V=\emptyset, \overline V\cap U=\emptyset$) such that $C_o^1\subset U, C_o^2\subset V$ since $A$ is Hausdorff. We claim that $C_\alpha \subset U\cup V$ for all $\alpha$ small enough. Otherwise, we could pick a sequence $x_{\alpha_i}\in C_{\alpha_i}\backslash (U\cup V)$ with $\lim_{i\to \infty}\alpha_i\to o$. Then $\lim_{i\to \infty} x_{\alpha_i}\in (U\cup V)^c\cap C_o$, contradicting the choice of $U,V$. This claim implies that $C_\alpha=(U\cap C_\alpha)\sqcup (V\cap C_\alpha)$.  By the assumption on connectivity of $C_\alpha$, for any fixed $\alpha\in \mathcal{I}\backslash o$ small, $C_\alpha\subset U$ or $C_\alpha\subset V$. Since $C_\alpha$ forms a decreasing family of subsets about $\alpha$,  actually either $C_\alpha\subset U$ for all $\alpha$ small enough, or $C_\alpha\subset V$ for all $\alpha$ small enough. We thus get $C_o\subset \overline U$ or $C_o\subset \overline V$, contradicting non-emptiness of $C_o^1$ and $C_o^2$.
\end{proof}

 \begin{rmk}
 If $A$ is not compact the conclusion is not true. Let $D_i:=\{(x,e^{-x})| x\in [0, i]\}\cup \{(x, 0)|x\in [0, i]\}\cup \{(i, y)|y\in [0, e^{-i}]\}\subset \mathbb{R}^2$, then the limiting set of the path-connected sets $C_i=\cup_{j\geq i} D_j$ is $\lim C_i=\{(x, e^{-x})|x\in [0, \infty)\}\cup \{(x,0)|x\in [0, \infty)\}$, which is not connected.
  \end{rmk}

\begin{rmk}(On ``refined GH topology''). A small issue that should be noticed is that in general, even for smooth KE Fano varieties, it may happen that two not biholomorphic KE Fano varieties are isometric as metric spaces (w.r.t. the metric structure induced by the KE metrics). In the smooth case it is easy to show, using holonomy and vanishing considerations, see Theorem $1.1.1$ in \cite{SpTh}, that this can happen only if the two varieties $X_1, X_2$ differ by complex conjugation in a product factor, i.e., $X_1\cong Y \times Z $ and $X_2\cong \overline{Y} \times Z$, with $Y,Z$ KE Fano manifolds with $Y \ncong \overline{Y}$ and, eventually, $Z=\{p\}$. This implies that the natural ``forgetful'' map from the set of biholomorphic isometry equivalence classes of KE Fano manifolds (which forms, set-theoretically, the moduli space $\mathcal{M}$ we are interested in) to the space of compact metric spaces (equipped with the GH topology) is not injective in general. Thus some care is required when we consider the GH compactification of the KE moduli space. This is the reason why we work on the Chow variety.

\end{rmk}

\subsection{Continuity Method to achieve Existence}

Let $(\mathcal X, \mathcal D)$ be as before. 

\begin{thm}\label{thm4-9}
Suppose $(X_0, (1-\beta_*)D_0)$ is K-polystable for some $\beta_*\in (0, 1]$, then there exists a unique weak conical K\"ahler-Einstein metric on $(X_0, D_0)$ with angle $2\pi\beta$ for $\beta\in (0, \beta_*]$.
\end{thm}

Define 
$$A:=\{\beta\leq \beta_*| \mbox{There exists a weak conical KE metric on } (X_0, D_0)\mbox{ with angle }2\pi\gamma \mbox{ for any }\kappa\leq\beta\}. $$
 As in the work of \cite{CDS1,CDS2, CDS3}, we will also use the method of deforming the cone angles. By Theorem \ref{thm1-3}, in order to establish Theorem \ref{thm4-9}, it suffices to show $A$ is both open and closed in $[\underline{\beta},\beta_*]$, where $\underline{\beta}$ is the number we obtained in Theorem \ref{thm1-3}.

\subsubsection{Openness}

For any $\beta(<\beta_*)\in A$, i.e. we have a weak conical K\"ahler-Einstein pair $(X_0, (1-\kappa)D_0, \omega_{0,\kappa})$ for any $\kappa\leq \beta$. We first show that $\beta_t> \beta$ for $|t|$ small enough. To see this, we use the lower semicontinuity of $\beta_t$ (Proposition \ref{Semi-C}). Suppose otherwise we have a subsequence $t_i\to 0$, $\beta_{t_i}\leq\beta$ and $\lim_{i\to \infty}\beta_{t_i}=\beta$. For each $i$, the weak conical K\"ahler-Einstein pair $(X_{t_i}, (1-\beta)D_{t_i}, \omega_{t_i,\beta})$ by sequence converges to a limit $(W_i,(1-\beta_{t_i})\Delta_i, \omega_i)$ as $\beta\to \beta_{t_i}$, with $Aut_0(W_i, \Delta_i)$ containing a non-trivial one parameter subgroup. By Lemma \ref{Uniqueness}, this sequence of limits must converge to $(X_0,(1-\beta)D_0, \omega_{0, \beta})$, contradicting the fact that $Aut_0(X_0, D_0)=\{1\}$.

Then by arguments similar to the proof of Lemma \ref{Uniqueness},  $(X_{t}, (1-\beta)D_{t}, \omega_{t,\beta})$ converges to $(X_0, (1-\beta)D_0, \omega_{0, \beta})$ in the Gromov-Hausdorff sense, and by \cite{CDS2},  the potential $\phi_{t,\beta}$ of $\omega_{t, \beta}$ relative to the induced Fubini-Study metric  $\omega_{t}$ with respect to the $L^2$ holomorphic embedding is uniformly bounded in $L^\infty$. This implies that $I_{\omega_{t}}(\phi_{t,\beta})$ is uniformly bounded.  This fact together with that $\beta_t>\beta$ implies that there is a $\tilde \beta>\beta$ such that $I_{\omega_t}(\phi_{t, \beta'})$ is uniformly bounded  for $\beta'\in [\beta, \tilde\beta]$ and for all $t$ sufficiently close to zero.   Otherwise we could find a subsequence $t_i\to 0, \beta<\kappa_i<\beta_{t_i}$ that converges to $\beta$ and a weak conical K\"ahler-Einstein pair $(X_{t_i}, (1-\kappa_i)D_{t_i}, \omega_{t_i, \kappa_i})$ with $I_{\omega_{t_i}}(\phi_{t_i, \kappa_i})-I_{\omega_{t_i}}(\phi_{t_i,\beta})=C$ for a fixed large constant $C$, this would lead to a contradiction by the uniqueness of weak conical K\"ahler-Einstein metrics and the arguments that we have used frequently previously.  From the uniform bounds of $I_{\omega_t}(\phi_{t, \beta'})$ it follows from Theorem \ref{thm3-1} that $(X_t, (1-\beta')D_t, \omega_{t,\beta'})$ converges by sequence to some weak conical K\"ahler-Einstein metric $(X_0, (1-\beta')D_0,\omega_{0, \beta'})$ as $t\to 0$.  This proves the openness of $A$.

\subsubsection{Closedness}

Take any sequence $\{\beta_j\}_{j=1,2,\cdots}\subset A$ which strictly increases to $\beta_\infty$. By the proof of Proposition \ref{Semi-C}, for any $j$, $\beta_t\geq \beta_j$ for $t$ small enough and the weak conical K\"ahler-Einstein pair $(X_0, (1-\beta_j)D_0, \omega_{0, \beta_j})$ is the limit of genuinely conical K\"ahler-Einstein metrics on the smooth fibers.  This enables us to take sequential Gromov-Hausdorff limit of $(X_0,(1-\beta_j)D_0, \omega_{0,\beta_j})$ from which we get a weak conical K\"ahler-Einstein pair $(Y,(1-\beta_\infty)\Delta,\omega)$. As in \cite{CDS3}, if $(Y, \Delta)$ is not isomorphic to $(X_0, D_0)$ then there is a test configuration for $(X_0, D_0)$ with central fiber $(Y, \Delta)$. This shows that $(X_0, (1-\beta_\infty)D_0)$ is not K-polystable. As in \cite{CDS3}, this implies that $(X_0, (1-\beta)D_0)$ is not K-polystable. Contradiction. \\

\subsection{Finishing Proof of Theorem \ref{thm1-1}:} 
By \cite{SW} and \cite{BBEGZ}, the automorphism $Aut(X_0, D_0)$ is discrete, the openness part of this continuity method holds for $\beta<1$ (we do not need to assume that $Aut(X_0)$ is discrete). This finishes the proof of the first part of Theorem \ref{thm1-1} on the existence of weak K\"ahler-Einstein metrics on $\mathbb{Q}$-Gorenstein smoothable K-polystable $\mathbb{Q}$-Fano variety $X_0$, and moreover, the weak K\"ahler-Einstein metric on $X_0$ is the Gromov-Hausdorff limit of conical K\"ahler-Einstein metrics on nearby smooth Fano manifolds. 

For the second part of Theorem \ref{thm1-1}, under the further assumption that $Aut(X_0)$ is discrete, this weak KE metric is unique by the uniqueness theorem \cite{BBEGZ}, denoted by $\omega_{0, KE}$. Since $(X_0, (1-\beta)D_0)$ is K-polystable for all $\beta\in [\underline{\beta}, 1]$, for any sequence of weak conical KE metrics $(X_0, (1-\beta_i)D_0, \omega_{0,\beta_i})$ with $\beta_i\to 1$, the $L^2$ embedding would converge to the $L^2$ embedding defined by  $(X_0,\omega_{0,KE})$. In particular, this shows that the $L^2$ embeddings of all the weak conical KE metrics $(X_0, (1-\beta)D_0, \omega_{0,\beta})$ are bounded (in the sense that the all the projective transformations with respect to the fixed $L^2$ embedding of $(X_0, \omega_{0, KE})$ are bounded), and the $I$ functional is uniformly bounded (with respect to the fixed K\"ahler metric $\omega_{0, FS}$). Then an argument similar to the \emph{openness} part in section 4.2.1 shows that the nearby $(X_t, D_t)$ admits conical KE metric with angle $2\pi$, which is smooth by the removable singularity theorem in \cite{CDS3}.  

\begin{rmk}
By Theorem \ref{thm4-9}, for a $\mathbb{Q}$-Gorenstein smoothable $\mathbb{Q}$-Fano variety $X_0$, the existence angle of \emph{weak conical K\"ahler-Einstein metric} along $(X_0, D_0)$ as above is an interval of one of the following forms:

\noindent 1. $(0,\beta_0)$ for some $\beta_0\leq1$ with $(X_0, (1-\beta_0)D_0)$ strictly K-semistable; \\
\noindent 2. $(0,1]$ if $X_0$ is K-polystable.
\end{rmk}

\section*{Appendix}\label{Appendix}
In this appendix we show the existence of a family  of ``nice'' divisors as claimed in Theorem \ref{thm1-3}. The proof is essentially based on Bertini type results (compare \cite{K}, Section $4$). 

\begin{prop}
Let $\mathcal{X}\rightarrow \Delta$ be a $\Q$-Gorenstein smoothing of a $\Q$-Fano variety $X_0$. Then, eventually shrinking the disc, there exists a divisor $\mathcal{D}\in |-\lambda K_{\mathcal{X}/ \Delta}|$ with $\lambda>1$  such that $D_t=\mathcal{D}|_{X_t}$ is smooth for $t\neq 0$ and $(X_t, (1-\beta)D_t)$ is KLT for any $\beta\in (0, 1]$ and $t \in \Delta$.
\end{prop}
\begin{proof} Fix a sufficiently big integer $\lambda>0$ so that   $-\lambda K_{{X}_t}$ are very ample line bundles (with vanishing cohomology) and such that the dimension of the linear systems $|-\lambda K_{{X}_t}|$ is constantly equal to $d=d(\lambda)$.  
We begin by showing that on the central fiber ${X}_0$ the \emph{generic} element in the linear system $|-\lambda K_{{X}_0}|$ gives rise to a KLT pair $(X_0, (1-\beta)D_0)$ for any $\beta\in (0, 1]$ . By lemma 4.7.1 in \cite{K}, under that hypothesis  that the linear systems $|- \lambda K_{{X}_0}|$ is base point free, we have $$ \mbox{discrep}({X}_0, (1-\beta) |- \lambda K_{{X}_0}| )=\mbox{discrep}({X}_0) >-1,$$ for $\beta \in (0,1]$, i.e., $({X}_0, (1-\beta) |- \lambda K_{{X}_0}|)$ is KLT. Now the desired statement follows immediately from theorem 4.8.2 in \cite{K}.

By our assumption on $\lambda$,  $\pi_! (\mathcal{O}(-\lambda K_{\mathcal{X}/\Delta}))=\pi_{\ast} (\mathcal{O}(-\lambda K_{\mathcal{X}/\Delta}))$ is a holomorphic vector bundle over $\Delta$ of dimension $d+1$, whose fibers are canonically identified with $H^0({X}_t, -\lambda K_{{X}_t})$.
By taking an holomorphic framing $F$ of $\pi_{\ast} (\mathcal{O}(-\lambda K_{\mathcal{X}/\Delta}))$, we may assume that our family embeds by complete linear systems in $\P^d \times \Delta$ so that the following diagram commutes:
$$
\xymatrix{
\mathcal{X} \ar@{^{(}->}[r]^F \ar[rd]_{\pi} &
\P^d\times \Delta\ar[d]^{pr_2} \\
& \Delta }
$$
From now on let us identify our family $\mathcal{X}$ with its image $F(\mathcal{X})$ inside $\P^d\times \Delta$. Let $D_0$ be any divisor in $|-\lambda K_{{X}_0}|$. By \emph{very ampleness},  there exists an hyperplane $H(D_0) \in {\P^d}^\ast$ such that $D_0= H(D_0) \cap \mathcal{X}_0$ (scheme theoretic). Now consider the pull-back (alias restriction) of $\mathcal{X}$ to $H(D_0) \times \Delta$, i.e., $$\mathcal{H}(D_0):=\mathcal{X}\times_{(\P^d \times \Delta)} (H(D_0) \times \Delta).$$
By definition $\mathcal{H}(D_0)$ is a divisor in $\mathcal{X}$ whose (scheme theoretic) intersection $\mathcal{H}(D_0)\cap \mathcal{X}_t$ is a divisor in $|-\lambda K_{{X}_t}|$ for all $t\in \Delta$. With the above in mind, define the map
$$ \begin{array}{cccc}
R_t: & |-\lambda K_{{X}_0}| & \longrightarrow  &| -\lambda K_{{X}_t}| \\ 
 & D_0 & \longmapsto & \mathcal{H}(D_0)\cap {X}_t.
\end{array}$$
It is evident that the above map is projectively linear and injective for $t$ sufficiently small, since, by construction of the embedding, the fibers ${X}_t$ are not contained in hyperplanes. Thus $R_t$ must be a bijection (since the dimension of the linear system is constant).

By hypothesis, $X_t$ is smooth for $t \neq 0$ and, by the usual Bertini Theorem, the generic element in the $\lambda$th-anticanonical linear system is also smooth. Denote with $S_t \subset | -\lambda K_{{X}_t}|$ the subset of smooth divisors in $X_t$ for $t\neq 0$ and denote with $K_0$ the subset of divisors $D_0$ in ${X}_0$  such that $({X}_0, (1-\beta) D_0)$ is KLT for all $\beta \in (0,1]$. Thanks to the first part of the argument we know that $K_0$ is almost all of $|-\lambda K_{{X}_0}|$. Hence it follows that $R_t(K_0) \cap S_t \neq \emptyset $.

Let us rephrase what we have just proved: there is a point $p\neq 0$ (which we can take arbitrarly closed to $0$) and a  flat divisor (actually many) in $\mathcal{X}$, which we will call $\mathcal{D}$, such that $({X}_0, (1-\beta) \mathcal{D}_0)$ is KLT  for all $\beta \in [0,1)$ and  ${D}_p= {X}_p \cap \mathcal{D}$ is smooth. Since the locus of points $p$ in the disc $\Delta$ where ${D}_t$ is singular is an analytic variety and we know that such locus is not the entire (smaller) disc, we conclude that this locus must consist of isolated points. In particular there exists an $r>0$ such that in the family above the $r$-disc $\Delta_r$ the origin is the only point where ${D}_t$ is singular. This conclude the argument.
\end{proof}
 
\begin{rmk} Even if not necessary for the arguments in tha paper, we remark that $\mathcal{D}$ may be chosen so that ${D}_0$ is  a normal variety with at worst LT singularities. The techniques of the proof are standard (compare \cite{K}), but we sketch the main arguments for completeness.  Since $-\lambda K_{{X}_0}$ is very ample and $X_0$ has log-terminal singularities, the generic element of the very ample linear system is normal. To see that that the generic section is also LT, take any log-resolution of ${X}_0$, $f:Y\rightarrow {X}_0$. Then the generic normal hyperplane section $S$ has smooth birational transform $T$ and $f:T\rightarrow S$ is a log-resolution of $S$. Since ${X}_0$ is LT, by Prop. 7.7 of \cite{K},  $\mbox{discrep}(S)\geq \mbox{discrep}({X}_0) > -1,$ hence $S$ is LT. \end{rmk}

\end{document}